\documentclass[12pt]{amsart}
\usepackage{amssymb,amsmath,amsthm}
\usepackage{tikz}
\usepackage{cancel}
\usepackage{mathtools,float}
\usetikzlibrary{shapes.geometric, intersections, calc}
\usepackage[T1]{fontenc}

\DeclareFontFamily{U}{mathb}{\hyphenchar\font45}
\DeclareFontShape{U}{mathb}{m}{n}{
      <5> <6> <7> <8> <9> <10> gen * mathb
      <10.95> mathb10 <12> <14.4> <17.28> <20.74> <24.88> mathb12
}{}
\DeclareSymbolFont{mathb}{U}{mathb}{m}{n}
\DeclareMathSymbol{\llcurly}{3}{mathb}{"CE}
\DeclareMathSymbol{\ggcurly}{3}{mathb}{"CF}
\DeclareFontFamily{U}{matha}{\hyphenchar\font45}
\DeclareFontShape{U}{matha}{m}{n}{
      <5> <6> <7> <8> <9> <10> gen * matha
      <10.95> matha10 <12> <14.4> <17.28> <20.74> <24.88> matha12
      }{}
\DeclareSymbolFont{matha}{U}{matha}{m}{n}
\DeclareMathSymbol{\curlywedge} {2}{matha}{"4E}
\DeclareMathSymbol{\curlyvee} {2}{matha}{"4F}

\newcommand{\cellsize}{13}
\newlength{\cellsz} 
\newsavebox{\cell}
\sbox{\cell}{\begin{picture}(\cellsize,\cellsize)
\put(0,0){\line(1,0){\cellsize}}
\put(0,0){\line(0,1){\cellsize}}
\put(\cellsize,0){\line(0,1){\cellsize}}
\put(0,\cellsize){\line(1,0){\cellsize}}
\end{picture}}
\newcommand\cellify[1]{\def\thearg{#1}\def\nothing{}%
\ifx\thearg\nothing
\vrule width0pt height\cellsz depth0pt\else
\hbox to 0pt{\usebox{\cell} \hss}\fi%
\vbox to \cellsz{
\vss
\hbox to \cellsz{\hss$#1$\hss}
\vss}}
\newcommand\tableau[1]{\vtop{\let\\\cr
\baselineskip -16000pt \lineskiplimit 16000pt \lineskip 0pt
\ialign{&\cellify{##}\cr#1\crcr}}}

\usetikzlibrary{arrows,shapes,automata,backgrounds,decorations,petri,positioning,patterns}
\usetikzlibrary{decorations.pathreplacing,calc}
\usetikzlibrary{decorations.pathmorphing}
\usetikzlibrary{decorations.markings}

\theoremstyle{plain}
\newtheorem{thm}{Theorem}[section]
\newtheorem{lem}[thm]{Lemma}

\newtheorem{cor}[thm]{Corollary}

\newtheorem{prop}[thm]{Proposition}
\theoremstyle{definition}
\newtheorem{defn}[thm]{Definition}
\newtheorem{remark}[thm]{Remark}
\newtheorem{ex}[thm]{Example}

\newtheorem*{repp@ex}{\repp@title (continued)}
\newcommand{\newreppex}[2]{
\newenvironment{repp#1}[1]{
 \def\repp@title{#2 \ref{##1}}
 \begin{repp@ex}}
 {\end{repp@ex}}}
\makeatother
\newreppex{ex}{Example}

\newenvironment{myquote}
   {\begin{list}{}{\setlength{\leftmargin}{.35in}\setlength{\rightmargin}{0in}}
    \item\relax}
   {\end{list}}

\newcommand{\mf}[1]{\mbox{$\mathfrak #1$}}
\newcommand{\gap}{{\sf gap}}
\newcommand{\gapp}{{\sf gap_{pos}}}
\newcommand{\gapv}{{\sf gap_{val}}}
\newcommand{\av}{\mathfrak{S}}
\renewcommand{\read}{{\sf{read}}}

\title{Reduced word manipulation: patterns and enumeration}

\author[Bridget Eileen Tenner]{Bridget Eileen Tenner$^{\dagger}$}
\address{Department of Mathematical Sciences, DePaul University, Chicago, IL 60614}
\email{bridget@math.depaul.edu}

\thanks{$^{\dagger}$ Research partially supported by a Simons Foundation Collaboration Grant for Mathematicians.}


\begin{document}

\begin{abstract}
We develop the technique of reduced word manipulation to give a range of results concerning reduced words and permutations more generally. We prove a broad connection between pattern containment and reduced words, which specializes to our previous work for vexillary permutations. We also analyze general tilings of Elnitsky's polygon, and demonstrate that these are closely related to the patterns in a permutation. Building on previous work for commutation classes, we show that reduced word enumeration is monotonically increasing with respect to pattern containment. Finally, we give several applications of this work. We show that a permutation and a pattern have equally many reduced words if and only if they have the same length (equivalently, the same number of $21$-patterns), and that they have equally many commutation classes if and only if they have the same number of $321$-patterns. We also apply our techniques to enumeration problems of pattern avoidance, and give a bijection between $132$-avoiding permutations of a given length and partitions of that same size, as well as refinements of this data and a connection to the Catalan numbers.\\

\end{abstract}

\maketitle

\section{Introduction}

The reduced words of a permutation $w$ are strings of positive integers that describe the ways to write $w$ as a product of adjacent transpositions. This data has connections to Coxeter groups of other types and an influence on the structure of various objects related to permutations. Additionally, the Bruhat order on the symmetric group can be defined in terms of reduced words, and thus they are key to understanding the structure of this important poset (see, for example, \cite{dyer, hultman1, hultman2, hultman vorwerk, ragnarsson-tenner1, tenner patt-bru} for progress on elucidating this architecture and that of related objects).

In \cite{tenner rdpp}, we revealed a powerful relationship between reduced words and permutation patterns in the case of a vexillary permutation. That result had a range of algebraic, topological, and enumerative applications, addressing questions about commutation classes of reduced words, tilings of certain families of polygons, and the topological structure of a family of algebraic objects. In the present work, we expand the manipulative techniques on reduced words that we debuted in \cite{tenner rdpp}. This will enable us to show that the main result of \cite{tenner rdpp} is a special case of a broader phenomenon. Moreover, we employ these methods and that general phenomenon to give enumerative results about reduced words and commutation classes of permutations and the patterns that contain them, extending Stanley's analysis of reduced word enumeration in \cite{stanley} and our own work in \cite{tenner patt-bru, tenner rdpp}.

There are many ways to represent permutations, including reduced words and one-line notation. The latter is the most natural presentation for questions about patterns, a concept that gained broad attention after work of Simion and Schmidt \cite{simion schmidt}, and that has since become popular in many guises (see, for example, Kitaev's text \cite{kitaev} and the various special journal issues devoted to the topic, including \cite{AOC2007}). The first hint of a direct connection between reduced words and permutation patterns appeared in \cite{bjs}, where Billey, Jockusch, and Stanley proved that $321$-avoidance was equivalent to avoiding $i(i\pm1)i$ factors in reduced words. We gave a broader relationship between reduced words and permutation patterns in \cite{tenner rdpp}, showing that a permutation $p$ is vexillary if and only if, roughly speaking, reduced words of $p$ can appear as isolated factors in reduced words of $w$, for all $w$ containing $p$. That result had a number of applications, notably involving tilings of the polygon $X(w)$, whose rhombic tilings were shown to be in bijection with commutation classes of reduced words of $w$ by Elnitsky in \cite{elnitsky}.

The ``manipulation'' in the title of this paper refers to the techniques we employ in this work. More precisely, our approach to analyzing features of a permutation (and its reduced words) will often include shortening it via a sequence of targeted position and value swaps, where these swaps correspond to right and left multiplication (respectively) by adjacent transpositions. This technique can be highly revealing about the reduced words of a permutation, as we shall see in subsequent sections, and can expose enumerative phenomena that have been hitherto unrecognized.

In Section~\ref{sec:definitions} of this paper, we discuss terminology and symbols relevant to this work, including examples of the basic objects under examination here. Section~\ref{sec:patt redwd} introduces the main manipulative techniques we apply to reduced words and gives the first main result of these efforts in Theorem~\ref{thm:patt redwd}. That theorem gives the broad relationship described earlier, between containment of a $p$-pattern and the influence of $p$'s reduced words on the reduced words of the larger permutation. The next two sections of the paper are concerned with applications of Theorem~\ref{thm:patt redwd} and, more generally, with the manipulative techniques used in the proof of that result. To be precise, Section~\ref{sec:comb apps} shows how pattern containment influences the tiling of a particular polygon used in the study of Coxeter groups, with the culminating result given in Corollary~\ref{cor:isolated paw patterns}. In Section~\ref{sec:enum apps}, we apply our manipulative techniques to enumerative questions about commutation classes of permutations and to the number of reduced words of a permutation as they relate to those of its patterns. Each of these quantities is monotonically increasing with respect to pattern containment (Proposition~\ref{prop:C mono} and Theorem~\ref{thm:R mono}), and we can describe exactly when equality occurs in each case (Theorems~\ref{thm:equal R} and~\ref{thm:equal C}). We also show how our methods can be employed to enumeration problems of pattern avoidance, and demonstrate this by giving a bijection between $132$-avoiding permutations of length $\ell$ and partitions of $\ell$ (Theorem~\ref{thm:132 by length}). This bijection has additional implications, related to the number of restricted partitions and to the Catalan numbers (Corollaries~\ref{cor:restricted partition} and~\ref{cor:refining catalan}).

\section{Definitions and notation}\label{sec:definitions}

We are concerned with the symmetric group, also known as the finite Coxeter group of type $A$. In this section, we introduce the relevant definitions and notation, and additional background on these objects can be found in the literature, such as in \cite{bjorner brenti, macdonald}.

We write $\mf{S}_n$ for the symmetric group on $\{1,\ldots,n\}$, and a permutation $w \in \mf{S}_n$ is represented in \emph{one-line notation} as the word $w = w(1)w(2)\cdots w(n)$.

The \emph{adjacent transpositions} in $\mf{S}_n$ are $\{\sigma_i: 1 \le i < n\}$, where $\sigma_i$ transposes $i$ and $i+1$ and fixes all other letters. These involutions generate $\mf{S}_n$, and obey the \emph{Coxeter relations}
\begin{eqnarray}
& \sigma_i\sigma_j = \sigma_j\sigma_i &\text{ if } |i-j| > 1, \text{ and}\label{eqn:commutation}\\
& \sigma_i\sigma_{i+1}\sigma_i = \sigma_{i+1}\sigma_i\sigma_{i+1} &\text{ for } 1 \le i \le n-2.\label{eqn:braid}
\end{eqnarray}
The relation in equation~\eqref{eqn:commutation} is a \emph{commutation}, and the relation in equation~\eqref{eqn:braid} is a \emph{braid}. Because $\{\sigma_i: 1 \le i < n\}$ generates $\mf{S}_n$, each $w \in \mf{S}_n$ can be written as a product of adjacent transpositions, where we view permutations as maps and thus $\sigma_iv$ transposes the positions of the values $i$ and $i+1$ in $v$, whereas $v\sigma_i$ transposes the values in positions $i$ and $i+1$ in $v$.

\begin{defn}
If $w = \sigma_{i_1} \cdots \sigma_{i_{\ell}}$ for $\ell$ minimal, then $\ell = \ell(w)$ is the \emph{length} of $w$, the product
$$\sigma_{i_1} \cdots \sigma_{i_{\ell}}$$
is a \emph{reduced decomposition} of $w$, and the string
$${\sf i_1\cdots i_{\ell}}$$
is a \emph{reduced word} of $w$. The set of all reduced words of $w$ is denoted $R(w)$.
\end{defn}

We will use sans-serif typeface for reduced words, thus distinguishing them from permutations in one-line notation.

Note that length can be calculated by counting inversions: $\ell(w) = \big|\{i<j : w(i) > w(j)\}\big|$.

Reduced decompositions and reduced words are in obvious bijection with each other, and the terms ``commutation'' and ``braid'' transfer to the context of reduced words in the natural way.

\begin{ex}\label{ex:3241}
Consider $3241 \in \mf{S}_4$. Its reduced decompositions are $\sigma_2\sigma_1\sigma_2\sigma_3$, $\sigma_1\sigma_2\sigma_1\sigma_3$, and $\sigma_1\sigma_2\sigma_3\sigma_1$, and thus
$$R(3241) = \{{\sf 2123}, {\sf 1213}, {\sf 1231}\}.$$
The first two elements of $R(3241)$ differ by a braid, while the latter two elements differ by a commutation.
\end{ex}

Reduced words are strings, and we will sometimes view them as nothing more than that.

\begin{defn}
A \emph{factor} in a string is a consecutive substring. \emph{Shifting} a string \emph{by $x$} means adding a fixed value $x$ to each symbol in the string. When specificity is not needed, we will suppress the ``by $x$.''
\end{defn}

A permutation in one-line notation is itself a string, and we will have use for the following concept.

\begin{defn}
Two strings, whose letters are drawn from a totally ordered set, are \emph{isomorphic} if their symbols appear in the same relative order. We use $\approx$ to denote this (equivalence) relation.
\end{defn}

\begin{ex}
$42135 \approx 82(-4)\pi\sqrt{95}$.
\end{ex}

The notion of isomorphic strings is central to the definition of permutation patterns.

\begin{defn}
Fix $p \in \mf{S}_k$. A permutation $w$ \emph{contains} a \emph{$p$-pattern}, denoted $p \prec w$, if $p \approx w(i_1)\cdots w(i_k)$ for some $i_1 < \cdots < i_k$. This $w(i_1)\cdots w(i_k)$ is an \emph{occurrence} of $p$, and we may denote $w(i_j)$ by $\langle p(j) \rangle$. If $w$ does not contain a $p$-pattern, then $w$ \emph{avoids} $p$ or is \emph{$p$-avoiding}, denoted $p \not\prec w$. For a set of patterns $P$, write
$$P \llcurly w \ \ \Longleftrightarrow \ \ p \prec w \text{ for all } p \in P.$$
\end{defn}

Pattern containment depends on there being at least one occurrence of the pattern, but the (positive) multiplicity of occurrences is unimportant.

\begin{ex}
The permutation $42135$ contains the pattern $213$, and it does so in five ways: $425 \approx 415 \approx 435 \approx 213 \approx 215$. The permutation $45132$ is $213$-avoiding, so $213\not\prec 45132$.
\end{ex}

Theorem~\ref{thm:patt redwd} of this article gives a potential relationship between elements of $R(p)$ and elements of $R(w)$ when $p \prec w$. That result, which will specialize to the main result of \cite{tenner rdpp}, involves an object that is most easily described as a set of the ``barred'' patterns introduced by West in \cite{west}.

\begin{defn}
A \emph{barred pattern} $\overline{p}$ is a permutation $p$ in which some letters are decorated by bars. A permutation $w$ \emph{contains} $\overline{p}$, denoted $\overline{p} \prec w$, if $w$ has an occurrence of the undecorated portion of $\overline{p}$ that is not also part of a larger $p$-pattern.
\end{defn}

\begin{ex}
Let $w = 42135$ and $\overline{p} = 3\overline{2}14$. Thus $3\overline{2}14 \prec 42135$ because $\langle 213 \rangle = 213$ is not part of any $3214$-pattern in $w$. We could not have drawn this conclusion from $\langle 213 \rangle = 415$, because of $\langle 3214 \rangle = 4215$ in $w$.
\end{ex}

Note that barred patterns can also be expressed in the more general language of mesh patterns, as introduced in \cite{branden claesson}.

Our work in \cite{tenner rdpp} gave a new definition of vexillary permutations. The feature of non-vexillary permutations that caused trouble was that a $2143$-pattern could be ``spread out,'' both in position and value by inserting a letter between the $1$ and the $4$ to yield the permutation $21354$. Let us capture this phenomenon more precisely.

\begin{defn}\label{defn:spread}
Fix $p \in \mf{S}_k$. A barred pattern $\overline{q}$ is a \emph{spread} of $p$ if its undecorated portion is isomorphic to $p$ itself and if $q \in \mf{S}_{k+1}$ is obtained by inserting a letter (barred in $\overline{q}$) to transform a $2143$-pattern in $p$ into a $21354$ pattern. In the degenerate case, when $p$ avoids $2143$, the only \emph{spread} of $p$ is $p$ itself. Write $p^+$ for the set of spreads of $p$.
\end{defn}

Note that a permutation may have more than one spread.

\begin{ex}
$251364^+ = \{261\overline{3}475, 2614\overline{3}75, 261\overline{4}375, 2613\overline{4}75\}$.
\end{ex}

\section{Influence of patterns on reduced words}\label{sec:patt redwd}

The main result of \cite{tenner rdpp} was a relationship between containing a pattern $p$ and implications for one's reduced words in terms of the reduced words of $p$. That relationship, presented here as Corollary~\ref{cor:rdpp main}, held if and only if $p$ was vexillary; that is, if and only if $p$ avoided $2143$. Vexillary permutations were introduced independently by Lascoux and Sch\"utzenberger in \cite{lascoux schutzenberger} and by Stanley in \cite{stanley}. There are several equivalent definitions of vexillarity, such as $2143$-avoidance, the main result of \cite{tenner rdpp}, and others as discussed in \cite{macdonald}.

In Theorem~\ref{thm:patt redwd} below, we establish that elements of $R(p)$ appear as particular factors in elements of $R(w)$ if and only if $p^+ \llcurly w$. Moreover, these factors should not be unduly influenced by their prefix or suffix, as expressed in the following definitions.

\begin{defn}\label{defn:isolation}
Fix $w \in \mf{S}_n$ and consider ${\sf s} \in R(w)$, where ${\sf s} = {\sf abc}$ is the concatenation of three (possibly empty) strings with ${\sf a} \in R(u)$ and ${\sf c} \in R(v)$. The factor ${\sf b}$ is \emph{isolated on $[m,m']$} in ${\sf s}$ if 
there are integers $m \le r \le t \le m'$ such that
\begin{enumerate}
\item the smallest letter appearing in ${\sf b}$ is $r$ and the largest letter appearing in ${\sf b}$ is $t$,
\item $\ell(u \sigma_i)> \ell(u)$ and $\ell(\sigma_i v) > \ell(v)$ for all $i \in [r,t]$, and
\item there are two particular increasing patterns in each of $u$ and $v$:
\begin{enumerate}
\item in $u$, an increasing sequence of length $r-m+1$ ends in position $r$, and an increasing sequence of length $m' - t + 1$ begins in position $t+1$, and
\item in $v$, an increasing sequence of length $r-m+1$ ends with the value $r$, and an increasing sequence of length $m'-t+1$ begins with the value $t+1$.
\end{enumerate}
\end{enumerate}
\end{defn}

While the concept of a factor is relatively common, isolation is unusual enough -- and important enough in the context of this work -- to merit some examples.

\begin{ex}
Consider the reduced word $\sf{s} = \sf{12325} \in R(243165)$.
\begin{enumerate}\renewcommand{\labelenumi}{(\alph{enumi})}
\item Define the factor $\sf{b} = \sf{2}$ so that $\sf{a} = \sf{1}$ and $\sf{c} = \sf{325}$. Thus $r = t = 2$. This $\sf{b}$ is not isolated on $[1,2]$ in $\sf{s}$ because, in the language of Definition~\ref{defn:isolation}, $u = 213456$ and $\ell(u\sigma_1) < \ell(u)$.
\item Now define the factor $\sf{b} = \sf{2}$ so that $\sf{a} = \sf{123}$ and $\sf{c} = \sf{5}$. Thus $r = t = 2$. This $\sf{b}$ is isolated on $[1,2]$ in $\sf{s}$ because $u = 234156$ and $v = 123465$, and both $\ell(u\sigma_i) > \ell(u)$ and $\ell(\sigma_iv) > \ell(v)$ for all $i \in [1,2]$.
\end{enumerate}
\end{ex}

Note that Definition~\ref{defn:isolation} fixes a mistake in \cite{tenner rdpp}, in which it was inadvertently assumed that patterns would not begin or end with fixed points. In the language of Definition~\ref{defn:isolation} above, we would have $m = r$ and $m' = t$ for such a situation, and so requirement (3) of the definition would be trivial. It should be noted that the applications of the main result of \cite{tenner rdpp} are unaffected by this correction, and many of them do not even involve patterns with initial or concluding fixed points. Also, if we were to think of permutations as actions on $\mathbb{Z}$ that involve only finitely many elements, then requirement (3) of Definition~\ref{defn:isolation} would again be trivial, because there would be increasing sequences of infinite length beginning or ending with any position or value.

\begin{ex}\
\begin{enumerate}\renewcommand{\labelenumi}{(\alph{enumi})}
\item Consider the reduced word $\sf{23} \in R(13425)$. Let $\sf{b} = 3$, $\sf{a} = 2$, and $\sf{c} = \emptyset$. Thus $r = t = 3$. Consider $m = 2$ and $m' = 4$. In the language of Definition~\ref{defn:isolation}, $u = 13245$ and $v = 12345$. First note that $\ell(u\sigma_3) > \ell(u)$ and $\ell(\sigma_3v) > \ell(v)$. Also, $u$ has an increasing sequence of length $r-m+1 = 2$ ending in position $r = 3$ (the sequence $12$), and an increasing sequence of length $m' - t + 1 = 2$ beginning in position $t + 1 = 4$ (the sequence $45$). Similarly, $v$ has the necessary increasing sequences. Thus this $\sf{b}$ is isolated on $[2,4]$ in $\sf{23}$.
\item In contrast, consider the reduced word $\sf{234} \in R(13452)$. Again let $\sf{b} = 3$ and $\sf{a} = 2$, and so we now have $\sf{c} = 4$. Again $r = t = 3$, and we continue to consider $m = 2$ and $m' = 4$. In the language of Definition~\ref{defn:isolation}, $u = 13245$ and $v = 12354$. As before, $\ell(u\sigma_3) > \ell(u)$ and $\ell(\sigma_3v) > \ell(v)$, and although $u$ has the required increasing sequences, the permutation $v$ does not; specifically, it has no increasing sequence whose minimum value is $t + 1 = 4$. Thus this $\sf{b}$ is not isolated on $[2,4]$ in $\sf{234}$.
\end{enumerate}
\end{ex}

In our usage, the range $[m,m']$ of Definition~\ref{defn:isolation} will be predetermined, as described below.

\begin{defn}
Fix $p \in \mf{S}_k$. A reduced word for $p$ that has been shifted by $x$ is \emph{isolated} in some ${\sf s}$ if it is isolated on $[x+1,x+k-1]$ in ${\sf s}$.
\end{defn}

The importance of isolation arises from the following result, which describes when pattern containment might be affected by products of adjacent transpositions.

\begin{lem}\label{lem:extending a pattern}
Fix permutations $p$ and $w$ for which $w$ has a $p$-pattern occupying positions $P$ and using values $V$. If $\{i, i+1\}\not\subseteq V$, then
$$p \prec \sigma_iw.$$
If $\{j, j+1\}\not\subseteq P$, then
$$p \prec w\sigma_j.$$
\end{lem}

\begin{proof}
If $\{i, i+1\}\not\subseteq V$, then $\sigma_iw$, which transposes the positions of the values $i$ and $i+1$ in $w$, has a $p$-pattern occupying positions $P$. If $\{j, j+1\}\not\subseteq P$, then $w\sigma_j$, which transposes the values in positions $j$ and $j+1$ in $w$, has a $p$-pattern using values $V$.
\end{proof}

The following example demonstrates necessity of the hypotheses in Lemma~\ref{lem:extending a pattern}.

\begin{ex}
Consider $231 \prec 4231$, for which $P = \{2,3,4\}$ and $V = \{1,2,3\}$. The permutation
$$\sigma_3(4231) = 3241$$
has a $231$-pattern in occupying positions $P$, and
$$(4231)\sigma_1 = 2431$$
has a $231$-pattern using values $V$. On the other hand, the following permutations are all $231$-avoiding:
\begin{eqnarray*}
\sigma_1(4231) &=& 4132,\\
\sigma_2(4231) &=& 4321,\\
(4231)\sigma_2 &=& 4321,\\
(4231)\sigma_3 &=& 4213.
\end{eqnarray*}
\end{ex}

The influence of isolation in Lemma~\ref{lem:extending a pattern} will be important to the proof of Theorem~\ref{thm:patt redwd}. Indeed, that argument will be constructive and will involve manipulation of the permutation $w$ via the following mechanism.

\begin{cor}\label{cor:isolation and patterns}
Fix permutations $p$ and $w$ and suppose that elements of $R(p)$ appear as shifted isolated factors in elements of $R(w)$. Then $p \prec w$. Moreover, if $w$ has a $p$-pattern in positions $P$ and using values $V$, then for $\{i, i+1\}\not\subseteq V$ and $\{j, j+1\}\not\subseteq P$,
\begin{itemize}
\item elements of $R(p)$ appear as shifted isolated factors in elements of $R(\sigma_iw)$, and
\item elements of $R(p)$ appear as shifted isolated factors in elements of $R(w\sigma_j)$.
\end{itemize}
\end{cor}

\begin{proof}
This can be proved by induction on $\ell(w) - \ell(p)$.
\end{proof}

The next lemma gives the basis of the inductive argument in the proof of Theorem~\ref{thm:patt redwd}.

\begin{lem}\label{lem:consecutive}
Suppose there is an occurrence of $p$ in $w$ that either occupies consecutive positions or uses consecutive values. Then elements of $R(p)$ appear as shifted isolated factors in elements of $R(w)$.
\end{lem}

\begin{proof}
Suppose that $p \in \mf{S}_k$ and $\langle p \rangle$ occupies positions $\{x+1,\ldots, x+k\}$ of $w$. Let $w'$ be the permutation obtained by writing the letters of $\langle p \rangle$ in increasing order and leaving all other letters fixed. Thus $\ell(w') = \ell(w) - \ell(p)$. Consider any ${\sf s} \in R(p)$ and ${\sf t} \in R(w')$, and let ${\sf s'}$ be the shift of ${\sf s}$ by $x$. Then the word ${\sf ts'}$ represents a product of adjacent transpositions that produces $w$. This word has length $\ell(w') + \ell(p) = \ell(w)$, so ${\sf ts'} \in R(w)$, as desired. If ${\sf s'}$ were not isolated in ${\sf ts'}$, then $w$ would not have a $p$-pattern in positions $\{x+1,\ldots, x+k\}$. Thus ${\sf s'}$ must be isolated.

The case of $\langle p \rangle$ using consecutive values is analogous, with the shifted reduced word of $p$ occurring as a factor on the left side in the reduced word of $w$.
\end{proof}

We are now ready for the main result of this section. Its proof is inductive, measured by how far the occurrences of $p$ are from satisfying the hypothesis of Lemma~\ref{lem:consecutive}. When $p^+ \llcurly w$, the proof constructs an element of $R(w)$ that contains an isolated shift of any $R(p)$ element as a factor.

\begin{thm}\label{thm:patt redwd}
Elements of $R(p)$ appear as shifted isolated factors in elements of $R(w)$ if and only if $p^+ \llcurly w$.
\end{thm}

\begin{proof}
Fix $p \in \mf{S}_k$.

First suppose that elements of $R(p)$ appear as shifted isolated factors in elements of $R(w)$. Then, by Corollary~\ref{cor:isolation and patterns}, we have $p \prec w$. Let ${\sf abc} \in R(w)$ be such a reduced word, where ${\sf b}$ is a shifted element of $R(p)$ that is isolated in ${\sf abc}$. Then ${\sf b} \in R(p')$, where
$$
p'(i) = 
\begin{cases}
p(i-x) + x & \text{if } i \in [x+1,x+k] \text{ and}\\
i & \text{otherwise,}
\end{cases}$$
for some $x$. The letters of ${\sf a}$ act on $p'$ by lengthening the permutation via transpositions of consecutive values. Because ${\sf b}$ is isolated and ${\sf a}$ is reduced, this procedure never transposes two values that are both in the pattern occurrence. Similarly, the reduced word ${\sf c}$ acts on that resulting permutation by lengthening it via transpositions of consecutive positions, never transposing two positions that are both in the pattern occurrence. The conclusion of these actions yields the permutation $w$. Consider any $2143$-pattern in $p$. The ``lengthening'' requirement of both ${\sf a}$ and ${\sf c}$ makes it impossible to move some $y \in (\langle 2 \rangle, \langle 3 \rangle)$ into the middle of a $2143$-pattern in $p$ unless it is swapping positions with some other $y' \in (\langle 2 \rangle,\langle 3 \rangle)$. Thus the only such values in $w$ must have been in $p$ itself, and so $p^+ \llcurly w$.

For the remainder of the proof, suppose that $p^+ \llcurly w$. We will induct on a statistic $\gap$ defined as follows. For a given $\langle p \rangle$ in $w$, let
$$\gapp(\langle p \rangle,w) = \Big(w^{-1}\big(\langle p(k) \rangle \big) - w^{-1}\big(\langle p(1) \rangle \big)\Big) - (k-1)$$
measure any excessive positional span of $\langle p \rangle$ in $w$, and
$$\gapv(\langle p \rangle,w) = \big(\langle k \rangle - \langle 1 \rangle\big) - (k-1).$$
measure any excessive value span. Both are nonnegative, and $\gapp(\langle p \rangle,w) = 0$ if and only if $\langle p \rangle$ appears in consecutive positions of $w$, while $\gapv(\langle p \rangle,w) = 0$ if and only if $\langle p \rangle$ uses consecutive values in $w$. Now set
$$\gap(p,w) = \min_{\langle p \rangle \text{ in } w} \Big\{ \gapp(\langle p \rangle,w) + \gapv(\langle p \rangle,w)\Big\}.$$

If $\gap(p,w) = 0$, then some $\langle p \rangle$ appears in consecutive positions in $w$ and with consecutive values. By Lemma~\ref{lem:consecutive}, then, the result holds.

Assume, inductively, that the result holds for all $q^+ \llcurly v$ with $\gap(q,v) < \gap(p,w)$. Fix an occurrence $\langle p \rangle$ in $w$ for which $\gap(p,w) = \gapp(\langle p \rangle,w) + \gapv(\langle p \rangle,w)$.

Call each $\langle p(i) \rangle$ a \emph{pattern entry}. Call $x \not \in \langle p \rangle$ a \emph{position gap} if it appears between $\langle p(1) \rangle$ and $\langle p(k) \rangle$ in the one-line notation of $w$, and a \emph{value gap} if $\langle 1 \rangle < x < \langle k\rangle$. If either $\gapp(\langle p \rangle,w)$ or $\gapv(\langle p \rangle,w)$ is $0$ then we can apply Lemma~\ref{lem:consecutive} and be done, so assume that both are positive. Fix $x$ to be the minimal position gap.

If $x$ is less than all pattern entries appearing to its left, then multiply $w$ on the right by adjacent transpositions to produce $w'$ in which $x$ has been shifted to the left of the $p$-pattern, and the rest of the one-line notation is unchanged. This $w'$ contains an occurrence of $p$ having the same values as $\langle p \rangle$ in $w$, but its positional span is one less than that of $\langle p \rangle$ in $w$. Thus $\gap(p,w') < \gap(p,w)$, and, because we never transposed two pattern entries in a single move, the result follows from the inductive hypothesis and Corollary~\ref{cor:isolation and patterns}.

If $x$ is greater than all pattern entries appearing to its right, then in fact all position gaps are greater than those pattern entries. Multiply $w$ on the right by adjacent transpositions to produce $w'$ in which the rightmost position gap has been shifted to the right of the $p$-pattern. Again, then, $\gap(p,w') < \gap(p,w)$, and the result follows from the inductive hypothesis and Corollary~\ref{cor:isolation and patterns}.

It remains to address when the following two situations occur simultaneously: a pattern entry less than $x$ appears to the left of $x$ in $w$, and a pattern entry greater than $x$ appears to the right of $x$ in $w$. Note that this implies existence of some $m$ for which $\langle m \rangle < x < \langle m+1 \rangle$.

Suppose that the pattern entries that are both less than $x$ and to its left appear in increasing order in $w$. The following procedure, which we call ($\dagger$), acts on a particular $p$-pattern and position gap in a permutation, and loops as necessary. Note that it never transposes two pattern entries in a single move.
\begin{itemize}
\item If $y = \langle m \rangle$ is to the left of $x$ in $w$ (necessarily the rightmost smaller pattern entry appearing to the left of $x$), then multiply $w$ on the right by adjacent transpositions to produce $w'$ in which $x$ has been shifted leftward until it abuts $y$. Each multiplication removes an inversion, shortening the permutation at each step. Moreover, the values of $\langle p \rangle$ still form a $p$-pattern in $w'$, as do the values of $\langle p \rangle$ with the exception of now using $x$ instead of $y$. If $y$ had been the leftmost pattern entry in $w$, then the revalued $p$-pattern has fewer position gaps than $\langle p \rangle$ had in $w$, so the result follows from the inductive hypothesis and Corollary~\ref{cor:isolation and patterns}. Otherwise, iterate ($\dagger$) using this new $p$-pattern in $w'$ and the position gap $y < x$.
\item If $y = \langle m \rangle$ is to the right of $x$ in $w$, then multiply $w$ on the left by elements of $\{\sigma_i : y \le i < x\}$ to produce $w'$ in which $\{y, y + 1, \ldots, x\}$ appear in increasing order and no other letters have moved. By definition of $x$ and $m$, this $w'$ contains a $p$-pattern in the same positions as $\langle p \rangle$ in $w$, and the only value that differs is some $y' > y$ now acting as ``$m$'' in the pattern. If $\gap(p,w') < \gap(p,w)$, then the result follows from the inductive hypothesis and Corollary~\ref{cor:isolation and patterns}. Otherwise, because the number of position gaps and value gaps have not changed with this revalued $p$-pattern, the measure $\gap(p,w')$ is obtained on it. Now iterate ($\dagger$) using this new $p$-pattern in $w'$ and the position gap $y < x$, which necessarily appears no further to the right in $w'$ than $x$ had appeared in $w$.
\end{itemize}

When all pattern entries that are both greater than $x$ and to its right form an increasing sequence in $w$, the argument is similar to ($\dagger$), with one additional step. Let $x'$ be the maximal position gap that does not appear to the left of $x$. Because $x' \ge x$, we in fact have that all pattern entries that are both greater than $x'$ and to its right appear in increasing order in $w$. Apply a procedure to $x'$ and its rightward larger pattern entries, analogous to the previous argument for $x$ and its leftward smaller pattern entries.

We have now addressed all scenarios except one: when neither $x$'s leftward smaller pattern entries nor its rightward larger pattern entries forms an increasing sequence; that is,
$$w = \cdots \langle p(a) \rangle \cdots \langle p(b) \rangle \cdots x \cdots \langle p(c) \rangle \cdots \langle p(d)\rangle \cdots,$$
for some
$$\langle p(b) \rangle < \langle p(a) \rangle < x < \langle p(d) \rangle < \langle p(c)\rangle.$$
But this means exactly that $p^+ \not\llcurly w$, which is a contradiction.
\end{proof}

There are two things in particular to observe about Theorem~\ref{thm:patt redwd}. First, if $p^+\llcurly w$ and ${\sf abc} \in R(w)$ with ${\sf b}$ a shifted isolated reduced word of $p$, then in fact we can replace ${\sf b}$ by that same shift of \emph{any} reduced word of $p$. Second, the proof of Theorem~\ref{thm:patt redwd} is constructive; that is, if $p^+ \llcurly w$ then the proof produces reduced words of $w$ that contain shifted reduced words of $p$ as isolated factors.

\section{Combinatorial influence of patterns on reduced words}\label{sec:comb apps}

One immediately corollary to Theorem~\ref{thm:patt redwd} recovers the main result of \cite{tenner rdpp}.

\begin{cor}[{\cite[Theorem 3.8]{tenner rdpp}}]\label{cor:rdpp main}
A permutation $p$ is vexillary if and only if for every $w \succ p$, elements of $R(p)$ appear as shifted isolated factors in elements of $R(w)$.
\end{cor}

\begin{proof}
Combine Theorem~\ref{thm:patt redwd} with the fact that $p$ is vexillary if and only if $p^+ = \{p\}$.
\end{proof}

Theorem~\ref{thm:patt redwd} also has applications for other objects, defined on a particular partition of the reduced words of a permutation.

\begin{defn}
Fix a permutation $w$ and define a relation $\sim$ on the set $R(w)$ such that ${\sf s} \sim {\sf t}$ if and only if ${\sf s}$ and ${\sf t}$ differ by a sequence of commutations. This $\sim$ is an equivalence relation, and the classes it defines are the \emph{commutation classes} of $w$, denoted $C(w)$.
\end{defn}

\begin{reppex}{ex:3241}
Observe that ${\sf 1213} \sim {\sf 1231}$. Thus the commutation classes of $3241 \in \mf{S}_4$ are $\{{\sf 2123}\}$ and $\{{\sf 1213}, {\sf 1231}\}$, and $|C(3241)| = 2$.
\end{reppex}

The commutation classes of a permutation do not take braids into account, but we can interpret the influence of those moves by means of a graph defined on commutation classes.

\begin{defn}
The \emph{graph of commutation classes} (or, simply, ``graph'' for our purposes) of a permutation $w$ has vertex set $C(w)$, and an edge between classes when representatives of those classes differ by a braid move.
\end{defn}

\begin{ex}
Because the reduced words ${\sf 2123}$ and ${\sf 1213}$ differ by a braid move, we have
$$\begin{minipage}{.75in}
$G(3241) =$
\end{minipage}
\begin{minipage}{1in}
\begin{tikzpicture}
\foreach \x in {(0,0),(0,1)} {\fill[black] \x circle (2pt);}
\draw (0,0) -- (0,1);
\draw (0,1) node[above] {$\{{\sf 2123}\}$};
\draw (0,0) node[below] {$\{{\sf 1213}, {\sf 1231}\}$};
\end{tikzpicture}
\end{minipage}.$$
\end{ex}

The sum of the letters in a reduced word is constant within a commutation class, and braid moves change the parity of this sum. Therefore the graph $G(w)$ is bipartite. It is also connected (see, for example, \cite{elnitsky}).

In fact, we can say more about $G(w)$ in light of Theorem~\ref{thm:patt redwd}.

\begin{cor}
If $p^+ \llcurly w$, then $G(p)$ is a subgraph of $G(w)$.
\end{cor}

Despite the natural definition of the classes $C(w)$, they are not especially well understood. For example, we cannot yet enumerate $|C(n(n-1)\cdots21)|$. In \cite{elnitsky}, Elnitsky defined a polygon $X(w)$ whose rhombic tilings are in bijection with elements of $C(w)$. This gives another perspective for working with commutation classes, and can be quite fruitful. Indeed, we used Elnitsky's polygon in \cite{tenner rdpp} to show that pattern containment affects the number of these classes.

\begin{prop}[{\cite[Theorem 5.10]{tenner rdpp}}]
If $p \prec w$ then $|C(p)| \le |C(w)|$.
\end{prop}

We now define Elnitsky's polygon, and refer the reader to \cite{elnitsky} for more details.

\begin{defn}
Fix a permutation $w \in \mf{S}_n$. \emph{Elnitsky's polygon}, denoted $X(w)$, is an equilateral $2n$-gon in which the sides are labeled $1, \ldots, n, w(n), \ldots, w(1)$ in counterclockwise order from the top, the left half (labeled $1, \ldots, n$) is convex, and sides are parallel if and only if they have the same label. Any polygon of this form is an \emph{Elnitsky polygon}.
\end{defn}

Note that Elnitsky polygons permit a type of degeneracy when $w$ has fixed points, in which circumstance the left and right borders of the polygon may coincide for one or more edges. Similarly, the left and right borders may intersect in a vertex if $\{w(1),\ldots,w(r)\} = \{1,\ldots,r\}$ for $w \in \mf{S}_n$ and some $r < n$.

The specific angles in $X(w)$ are unimportant. One could also replace the equilateral requirement by a rule that all parallel sides be congruent, but this adds no complexity to the object or the tilings of it that we will consider.

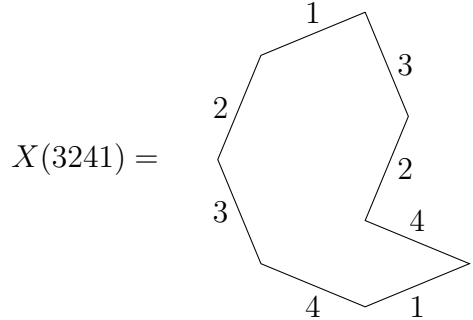
\begin{ex}
$$\begin{minipage}{1in}
$X(3241) = $
\end{minipage}
\begin{minipage}{2in}
\begin{tikzpicture}[scale=1.5]
\draw (0,0) coordinate (a) -- ++(-157.5:1) coordinate (b) -- ++(-112.5:1) coordinate (c) -- ++(-67.5:1) coordinate (d) -- ++(-22.5:1) coordinate (e) -- ++(22.5:1) coordinate (f) -- ++(157.5:1) coordinate (g) -- ++(67.5:1) coordinate (h) -- ++(112.5:1);
\draw ($(a)!0.5!(b)$) node[above] {$1$};
\draw ($(b)!0.5!(c)$) node[left] {$2$};
\draw ($(c)!0.5!(d)$) node[left] {$3$};
\draw ($(d)!0.5!(e)$) node[below] {$4$};
\draw ($(e)!0.5!(f)$) node[below] {$1$};
\draw ($(f)!0.5!(g)$) node[above] {$4$};
\draw ($(g)!0.5!(h)$) node[right] {$2$};
\draw ($(h)!0.5!(a)$) node[right] {$3$};
\end{tikzpicture}
\end{minipage}$$
\end{ex}

Elnitsky looked at rhombic tilings of $X(w)$.

\begin{defn}
Fix a permutation $w$. Let $T(w)$ be the set of rhombic tilings of $X(w)$ where all edges are congruent and parallel to edges of $X(w)$.
\end{defn}

Define a graph with vertex set $T(w)$ and an edge between two tilings if they differ only in the tiling of a single sub-hexagon, as in Figure~\ref{fig:hexagon}. 
\begin{figure}[htbp]
\begin{tikzpicture}
\draw (0,0) coordinate (a) -- ++(30:1) coordinate (b) -- ++(90:1) coordinate (c) -- ++(150:1) coordinate (d) -- ++(210:1) coordinate (e) -- ++(270:1) coordinate (f) -- (a);
\draw (d) -- ++(-90:1) coordinate (g) -- (b);
\draw (g) -- (f);
\draw (1.5,1) node {$\leftrightsquigarrow$};
\draw (3,0) coordinate (h) -- ++(30:1) coordinate (i) -- ++(90:1) coordinate (j) -- ++(150:1) coordinate (k) -- ++(210:1) coordinate (l) -- ++(270:1) coordinate (m) -- (h);
\draw (h) -- ++(90:1) coordinate (n) -- (j);
\draw (n) -- (l);
\end{tikzpicture}
\caption{The two ways to tile a sub-hexagon in Elnitsky's polygon.}\label{fig:hexagon}
\end{figure}
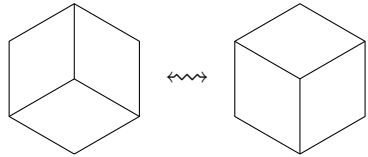
Elnitsky proved in \cite{elnitsky} that this graph is in bijection with the graph $G(w)$ of commutation classes. In particular, this means that $T(w)$ is in bijection with $C(w)$. The details of Elnitsky's bijection are useful to this work, and we review them here. Note that this is a slight variation of Elnitsky's original bijection.

\begin{prop}[{\cite[Theorem 2.2]{elnitsky}}]\label{prop:elnitsky}
Fix $w \in \mf{S}_n$ and $T \in T(w)$. Label the rhombi in $T$ by $1, 2, \ldots$ from right to left, always labeling a rhombus that shares two edges with the rightmost border of $X(w)$ or with tiles that have already been labeled. Use this labeling to create a string ${\sf s_{\ell(w)} \cdots s_2s_1}$, where ${\sf s_i = j}$ if the edges of tile $i$ are $j$th and $(j+1)$st in a path of length $n$ from the topmost to the bottommost vertex in $X(w)$. The map $T \rightarrow {\sf s_{\ell(w)} \cdots s_2s_1}$ gives a bijection between $T(w)$ and $C(w)$ because a given $T$ produced all reduced words within a single commutation class of $w$.
\end{prop}

\begin{ex}
Elnitsky's bijection for $3241 \in \mf{S}_4$ is
$$
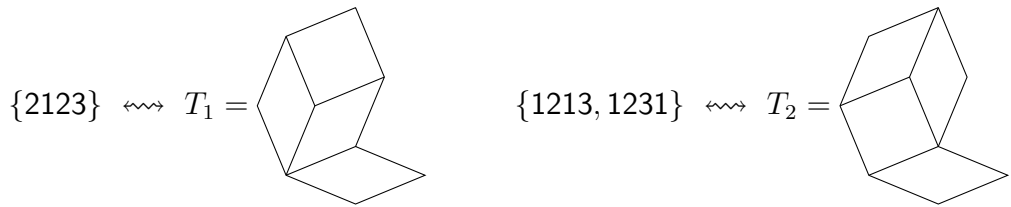
\begin{minipage}{1.3in}
$\{{\sf 2123}\} \ \leftrightsquigarrow \ T_1 =$
\end{minipage}
\begin{minipage}{1.35in}
\begin{tikzpicture}
\draw (0,0) coordinate (a) -- ++(-157.5:1) coordinate (b) -- ++(-112.5:1) coordinate (c) -- ++(-67.5:1) coordinate (d) -- ++(-22.5:1) coordinate (e) -- ++(22.5:1) coordinate (f) -- ++(157.5:1) coordinate (g) -- ++(67.5:1) coordinate (h) -- ++(112.5:1);
\draw (d) -- (g);
\draw (d) -- ++(67.5:1) coordinate (i) -- (b);
\draw (i) -- (h);
\end{tikzpicture}
\end{minipage}
\begin{minipage}{1.7in}
$\{{\sf 1213}, {\sf 1231}\} \ \leftrightsquigarrow\ T_2 =$
\end{minipage}
\begin{minipage}{1.35in}
\begin{tikzpicture}
\draw (0,0) coordinate (a) -- ++(-157.5:1) coordinate (b) -- ++(-112.5:1) coordinate (c) -- ++(-67.5:1) coordinate (d) -- ++(-22.5:1) coordinate (e) -- ++(22.5:1) coordinate (f) -- ++(157.5:1) coordinate (g) -- ++(67.5:1) coordinate (h) -- ++(112.5:1);
\draw (d) -- (g);
\draw (c) -- ++(22.5:1) coordinate (i) -- (a);
\draw (i) -- (g);
\end{tikzpicture}
\end{minipage}$$
because of $T_1$'s unique labeling
$$
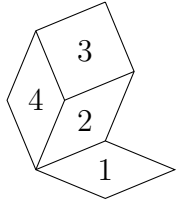
\begin{tikzpicture}
\draw (0,0) coordinate (a) -- ++(-157.5:1) coordinate (b) -- ++(-112.5:1) coordinate (c) -- ++(-67.5:1) coordinate (d) -- ++(-22.5:1) coordinate (e) -- ++(22.5:1) coordinate (f) -- ++(157.5:1) coordinate (g) -- ++(67.5:1) coordinate (h) -- ++(112.5:1);
\draw (d) -- (g);
\draw (d) -- ++(67.5:1) coordinate (i) -- (b);
\draw (i) -- (h);
\draw ($.25*(d)+.25*(e)+.25*(f)+.25*(g)$) node {$1$};
\draw ($.25*(d)+.25*(g)+.25*(h)+.25*(i)$) node {$2$};
\draw ($.25*(a)+.25*(b)+.25*(i)+.25*(h)$) node {$3$};
\draw ($.25*(b)+.25*(c)+.25*(d)+.25*(i)$) node {$4$};
\end{tikzpicture}$$
corresponding to $\{{\sf 2123}\}$, and the two possible labelings of $T_2$
$$
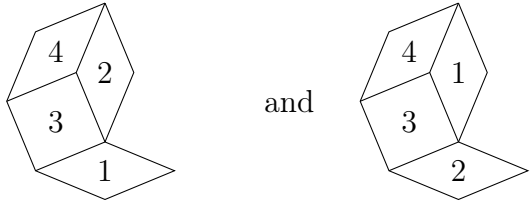
\begin{minipage}{1.35in}
\begin{tikzpicture}
\draw (0,0) coordinate (a) -- ++(-157.5:1) coordinate (b) -- ++(-112.5:1) coordinate (c) -- ++(-67.5:1) coordinate (d) -- ++(-22.5:1) coordinate (e) -- ++(22.5:1) coordinate (f) -- ++(157.5:1) coordinate (g) -- ++(67.5:1) coordinate (h) -- ++(112.5:1);
\draw (d) -- (g);
\draw (c) -- ++(22.5:1) coordinate (i) -- (a);
\draw (i) -- (g);
\draw ($.25*(g)+.25*(h)+.25*(a)+.25*(i)$) node {$2$};
\draw ($.25*(d)+.25*(e)+.25*(f)+.25*(g)$) node {$1$};
\draw ($.25*(c)+.25*(d)+.25*(g)+.25*(i)$) node {$3$};
\draw ($.25*(a)+.25*(b)+.25*(c)+.25*(i)$) node {$4$};
\end{tikzpicture}
\end{minipage}
\begin{minipage}{.5in}
\text{and}
\end{minipage}
\begin{minipage}{1.35in}
\begin{tikzpicture}
\draw (0,0) coordinate (a) -- ++(-157.5:1) coordinate (b) -- ++(-112.5:1) coordinate (c) -- ++(-67.5:1) coordinate (d) -- ++(-22.5:1) coordinate (e) -- ++(22.5:1) coordinate (f) -- ++(157.5:1) coordinate (g) -- ++(67.5:1) coordinate (h) -- ++(112.5:1);
\draw (d) -- (g);
\draw (c) -- ++(22.5:1) coordinate (i) -- (a);
\draw (i) -- (g);
\draw ($.25*(d)+.25*(e)+.25*(f)+.25*(g)$) node {$2$};
\draw ($.25*(g)+.25*(h)+.25*(a)+.25*(i)$) node {$1$};
\draw ($.25*(c)+.25*(d)+.25*(g)+.25*(i)$) node {$3$};
\draw ($.25*(a)+.25*(b)+.25*(c)+.25*(i)$) node {$4$};
\end{tikzpicture}
\end{minipage}$$
corresponding to $\{{\sf 1213}, {\sf 1231}\}$.
\end{ex}

The breadth of Theorem~\ref{thm:patt redwd} enables us to use more generic tiles in Elnitsky's polygon.

\begin{defn}
Fix a permutation $w$ and tile $X(w)$ by other Elnitsky polygons, always oriented so that the lefthand path from the highest to the lowest vertex is convex. Let $T^*(w)$ be the set of such tilings, called \emph{paw tilings} of $X(w)$, and call each tile in $T \in T^*(w)$ a \emph{paw}.
\end{defn}

Note that $T(w) \subseteq T^*(w)$ because each rhombus is an $X(21)$-paw.

\begin{ex}\label{ex:paw tiling}
Figure~\ref{fig:paw tiling} gives an element of $T^*(352641)$. It has two $X(21)$-paws, one $X(312)$-paw, and one $X(3421)$-paw. All tile edges have been labeled for clarity.
\begin{figure}[htbp]
$$\begin{tikzpicture}[scale=1.5]
\draw (0,0) coordinate (a) -- ++(-165:1) coordinate (b) -- ++(-135:1) coordinate (c) -- ++(-105:1) coordinate (d) -- ++(-75:1) coordinate (e) -- ++(-45:1) coordinate (f) -- ++(-15:1) coordinate (g) -- ++(15:1) coordinate (h) -- ++(105:1) coordinate (i) -- ++(165:1) coordinate (j) -- ++(45:1) coordinate (k) -- ++(135:1) coordinate (l) -- ++(75:1);
\draw (d) -- ++(-45:1) coordinate (m) -- (j);
\draw (m) -- (f);
\draw (g) -- ++(105:1) coordinate (n) -- (i);
\draw ($.5*(a) + .5*(b)$) node[above] {$1$};
\draw ($.5*(c) + .5*(b)$) node[above left] {$2$};
\draw ($.5*(c) + .5*(d)$) node[left] {$3$};
\draw ($.5*(e) + .5*(d)$) node[left] {$4$};
\draw ($.5*(e) + .5*(f)$) node[below left] {$5$};
\draw ($.5*(g) + .5*(f)$) node[below] {$6$};
\draw ($.5*(g) + .5*(h)$) node[below] {$1$};
\draw ($.5*(i) + .5*(h)$) node[above right] {$4$};
\draw ($.65*(i) + .35*(j)$) node[above] {$6$};
\draw ($.5*(k) + .5*(j)$) node[ right] {$2$};
\draw ($.5*(k) + .5*(l)$) node[above right] {$5$};
\draw ($.5*(a) + .5*(l)$) node[right] {$3$};
\draw ($.5*(m) + .5*(d)$) node[above right] {$5$};
\draw ($.5*(m) + .5*(j)$) node[above] {$1$};
\draw ($.5*(m) + .5*(f)$) node[right] {$4$};
\draw ($.5*(g) + .5*(n)$) node[right] {$4$};
\draw ($.5*(i) + .5*(n)$) node[below] {$1$};
\end{tikzpicture}$$
\caption{A paw tiling of $X(352641)$, described in Example~\ref{ex:paw tiling}.}\label{fig:paw tiling}
\end{figure}
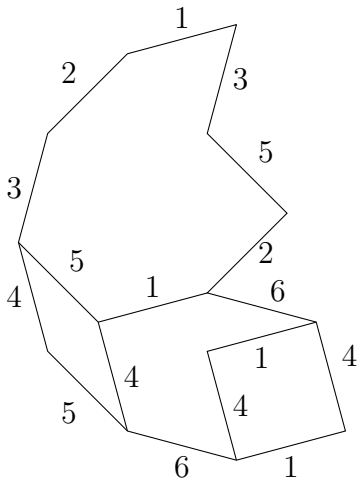
\end{ex}

In any tiling of $X(w)$, the \emph{edges} of a tile refer to its parallel edges along the border of $X(w)$.

The procedure described in Elnitsky's bijection (Proposition~\ref{prop:elnitsky}) builds a reduced word in a single direction. More precisely, labeling rhombi from the righthand side of $X(w)$ corresponds to multiplying $w$ on the right by adjacent transpositions (equivalently, reducing the length of $w$ by swapping adjacent positions). Recall the algorithm laid out in the proof of Theorem~\ref{thm:patt redwd}. Whenever possible, the statistic $\gap(p,w)$ was reduced by means of position swaps (that is, right multiplication by adjacent transpositions). This was done intentionally, to match the rightward favoritism of Elnitksy's bijection. However, there were two scenarios in the proof of Theorem~\ref{thm:patt redwd} that required value swaps (that is, left multiplication). These were, using the terminology of that proof, the second bullet point of the loop $(\dagger)$, and its analogue when the leftward smaller pattern entries to $x$ have a descent, but the rightward larger ones do not.

\begin{defn}
Suppose that $p \prec w$, and let $S$ be a set of values forming a $p$-pattern in $w$. Call $(p,w)$ a \emph{value-stable} pair if $w = w'v$ for some permutations $w'$ and $v$ such that
\begin{itemize}
\item $\ell(w) = \ell(w') + \ell(v)$ and
\item a $p$-pattern occurs in consecutive positions of $w'$, and the set of values forming that pattern in $w'$ is $S$.
\end{itemize}
The $p$-pattern formed by $S$ is the \emph{value-stable occurrence} of the pair.
\end{defn}

Note that requiring $(p,w)$ to be value-stable is more restrictive than requiring that $p^+ \llcurly w$.

\begin{ex}\ 
\begin{itemize}
\item The pair $(352641,3421)$ is value-stable because $352641 = (352164)(123564)$. The value-stable occurrence that appears consecutively in $352164$, which is also present in $352641$, is $3521$.
\item The pair $(241365,21354)$ is not value-stable.
\end{itemize}
\end{ex}

Using Elnitsky's bijection, Theorem~\ref{thm:patt redwd} has an immediate implication for the paw tilings in the context of value-stable pairs.

\begin{cor}\label{cor:paw patterns}
If $(p,w)$ is value-stable, then there is a paw tiling of $X(w)$ that contains an $X(p)$-paw.
\end{cor}

Unfortunately, the converse to Corollary~\ref{cor:paw patterns} is false. For example, there is a paw tiling of the convex hexagon $X(321)$ that contains an $X(312)$-paw, as shown in Figure~\ref{fig:non isolated paw}.
\begin{figure}[htbp]
\begin{tikzpicture}
\draw (0,0) coordinate (a) -- ++(30:1) coordinate (b) -- ++(90:1) coordinate (c) -- ++(150:1) coordinate (d) -- ++(210:1) coordinate (e) -- ++(-90:1) coordinate (f) -- (a);
\draw (a) -- ++(90:1) coordinate (g) -- (c);
\end{tikzpicture}
\caption{Tiling $X(321)$ with an $X(312)$-paw and a rhombus.}\label{fig:non isolated paw}
\end{figure}
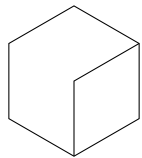
To understand when the converse to Corollary~\ref{cor:paw patterns} might be true, we need a notion of ``isolation'' for paws.

\begin{defn}
Consider a permutation $w$ and an $X(p)$-paw in some $T \in T^*(w)$. Let $S \subseteq T^*(w)$ be the paw tilings containing this particular $X(p)$-paw in this particular position, and tiling the rest of $X(w)$ by rhombi. If no element of $S$ has a rhombus sharing two edges with the righthand side of the $X(p)$-paw, then the $X(p)$-paw is \emph{isolated} in $T$.
\end{defn}

Observe that an $X(k\cdots21)$-paw is necessarily isolated because no rhombus can share two edges with its convex righthand side.

\begin{reppex}{ex:paw tiling}
The $X(3421)$-paw in Figure~\ref{fig:paw tiling} is isolated, but the $X(312)$-paw is not because of its potential (in fact, required) adjacency to the rhombus with edges $\{1,4\}$.
\end{reppex}

The factor isolation described in Theorem~\ref{thm:patt redwd} reveals, now, a biconditional analogue to Corollary~\ref{cor:paw patterns}. In other words, Corollary~\ref{cor:isolated paw patterns} describes exactly when we can fit an isolated $X(p)$-paw into a paw tiling of $X(w)$.

\begin{cor}\label{cor:isolated paw patterns}
There is a paw tiling of $X(w)$ containing an isolated $X(p)$-paw if and only if the pair $(p,w)$ is value-stable. Moreover, the edges of this isolated $X(p)$-paw are $\{i_1,\ldots,i_k\}$ if and only if $\{i_1,\ldots,i_k\}$ form a value-stable occurrence of $p$ in $w$.
\end{cor}

Because the permutation $k\cdots 21$ is vexillary, and because $X(k\cdots 21)$ is a convex $2k$-gon, we used Corollary~\ref{cor:rdpp main} to study zonotopal tilings of $X(w)$ in \cite{tenner rdpp}; that is, tilings by convex paws. In particular, there is a zonotopal tiling of $X(w)$ containing a $2k$-gon with sides labeled $i_1 < \cdots < i_k$ if and only if $i_k \cdots i_1$ form a $(k\cdots 21)$-pattern in $w$ \cite[Theorem~6.4]{tenner rdpp}. Theorem~\ref{thm:patt redwd} and Corollary~\ref{cor:isolated paw patterns} here now allow us to generalize that result completely.

\section{Enumerative applications}\label{sec:enum apps}

We now turn our attention and techniques to enumerative questions. In particular, we give applications of this work in two directions -- an analysis of the influence of permutation patterns on $|R(w)|$ and $|C(w)|$, and a sampling of how it can be applied to enumerative questions about pattern avoidance. That latter discussion gives an elegant bijection between partitions and a class of pattern avoiding permutations, as well as a refinement of the Catalan numbers. Moreover, and perhaps most interestingly, it demonstrates how our work here can give a new framework for analyzing problems of pattern avoidance enumeration.

\subsection{Enumerative influence of patterns on reduced words}\

In \cite{stanley}, Stanley showed that $|R(w)|$ is a linear combination of the number of standard Young tableaux of certain shapes. For example, if $w$ is vexillary, then $|R(w)|$ is equal to the number of standard Young tableaux of a single shape $\lambda(w)$. In \cite{billey pawlowski}, Billey and Pawlowski defined further classes of permutations based on how many shapes have nonzero coefficients in the sum. While these collections $R(w)$ can be enumerated, it is often complicated to do so. Moreover, the results do not give any indication of how $|R(p)|$ and $|R(w)|$ might be related (if at all) when $p \prec w$, whereas Theorem~\ref{thm:patt redwd} suggests that indeed some relationship is likely. Similarly, as mentioned earlier in this article, the number $|C(w)|$ of commutation classes of a permutation is very poorly understood outside of a few special cases.

In this section, we look at both $|R(w)|$ and $|C(w)|$ from the perspective of pattern avoidance. We will show that they are both monotonically increasing with respect to pattern containment, and we will completely characterize the $p\prec w$ for which equality is maintained, in each case.

Certainly Theorem~\ref{thm:patt redwd} implies the following result.

\begin{cor}\label{cor:counting R(w) for w with p+}
If $p^+ \llcurly w$ then $|R(p)| \le |R(w)|$.
\end{cor}

We can actually strengthen Corollary~\ref{cor:counting R(w) for w with p+} to a property about any $p \prec w$, regardless of whether $w$ contains all of $p^+$.

Consider the following algorithm, which we defined in \cite{tenner rdpp}.

\begin{myquote}

\setlength{\leftmargini}{0in}
\setlength{\leftmarginii}{.18in}
\setlength{\leftmarginiii}{.18in}
\setlength{\leftmarginiv}{.18in}

\noindent \textsf{\textbf{Algorithm} MONO}

\nopagebreak[4]

\noindent \textsf{INPUT: $p \prec w$ and $T \in T(p)$ labeled as described in Proposition~\ref{prop:elnitsky}.}

\noindent \textsf{OUTPUT: $T' \in T(w)$.}

\smallskip

\begin{enumerate}
\setcounter{enumi}{-1}
\renewcommand{\labelenumi}{\textsf{Step \arabic{enumi}}.}

\item Set $w_{0}:=w$, $p_{0}:=p$, $T_{0}:=T$, $T'_{0}:=\emptyset$, and $i:=0$.
\item\label{monostart} If $p_{i} = e$, then let $T'_{i+1}$ be $T'_{i}$ together with any tiling of $X(w_{i})$.  \textsf{OUTPUT $T'_{i+1}$}.
\item Set $j_{i}$ so that tile $i$ has edges $p_{i}(j_{i}) > p_{i}(j_{i}+1)$.
\item Define $r$ and $s$ so that $w_{i}(r) = \langle p_{i}(j_{i}) \rangle$ and $w_{i}(s) = \langle p_{i}(j_{i}+1) \rangle$.
\item Let $w_{i+1}$ be obtained from the one-line notation of $w_{i}$ be writing $\{w_{i}(r), \ldots, w_{i}(s)\}$ in increasing order and leaving all other values unchanged.
\item The shapes $X(w_{i+1})$ and $X(w_{i})$ differ in a paw whose lefthand (respectively, righthand) boundary is part of the righthand boundary of $X(w_{i+1})$ (respectively, $X(w_{i})$). Let $t_{i}$ be a rhombic tiling of this paw, and define $T'_{i+1}$ to be $T'_{i}$ together with $t_{i}$.
\item Set $i:=i+1$ and \textsf{GOTO Step~\ref{monostart}}.
\end{enumerate}
\end{myquote}

This {\sf MONO} gives an injection $T(p) \hookrightarrow T(w)$, yielding the following result.

\begin{prop}[{\cite[Theorem 5.10]{tenner rdpp}}]\label{prop:C mono}
If $p \prec w$ then $|C(p)| \le |C(w)|$.
\end{prop}

In fact, an analogous property holds for reduced words.

\begin{thm}\label{thm:R mono}
If $p \prec w$ then $|R(p)| \le |R(w)|$.
\end{thm}

\begin{proof}
Suppose $p \prec w$ and fix a tiling $T \in T(p)$, with $T' \in T(w)$ a corresponding tiling from {\sf MONO}. This $T$ describes a commutation class of reduced words that arise from labeling $T$, and such a labeling depends on choices about the order in which commuting tiles are labeled via Elnitsky's procedure. Consider one such labeling of $T$. It induces a labeling on the tiles of $T'$ via {\sf MONO}, where any collection of rhombi in $T'$ that come from a single (or empty, as in {\sf Step~\ref{monostart}} of the algorithm) tile in $T$ are labeled consecutively via Elnitsky's procedure. This gives an injection $R(p) \hookrightarrow R(w)$, and thus $R(p) \le R(w)$.
\end{proof}

Figure~\ref{fig:R mono} gives an example of the injection described in the proof of Theorem~\ref{thm:R mono}.

\begin{figure}[htbp]
\begin{minipage}{1.5in}
$$\begin{tikzpicture}
\draw (0,0) coordinate (a) 
-- ++(-162:1) coordinate (b) 
-- ++(-126:1) coordinate (c) 
-- ++(-90:1) coordinate (d) 
-- ++(-54:1) coordinate (e) 
-- ++(-18:1) coordinate (f)
-- ++(90:1) coordinate (g)
-- ++(126:1) coordinate (h)
-- ++(18:1) coordinate (i)
-- ++(54:1) coordinate (j) -- (a);
\draw (d) -- ++(-18:1) coordinate (k) -- (f);
\draw (k) -- (h) -- (c);
\draw (b) -- ++(-18:1) coordinate (l) -- (h);
\draw (l) -- (j);
\draw ($.25*(l)+.25*(h)+.25*(i)+.25*(j)$) node {$1$};
\draw ($.25*(l)+.25*(a)+.25*(b)+.25*(j)$) node {$2$};
\draw ($.25*(f)+.25*(g)+.25*(h)+.25*(k)$) node {$3$};
\draw ($.25*(b)+.25*(c)+.25*(h)+.25*(l)$) node {$4$};
\draw ($.25*(c)+.25*(d)+.25*(h)+.25*(k)$) node {$5$};
\draw ($.25*(d)+.25*(e)+.25*(f)+.25*(k)$) node {$6$};
\end{tikzpicture}$$
\end{minipage}
\begin{minipage}{.5in}
$$\longrightarrow$$
\end{minipage}
\begin{minipage}{2.5in}
$$\begin{tikzpicture}
\draw[ultra thick] (0,0) coordinate (a) 
-- ++(-167:1) coordinate (b) 
-- ++(-141:1) coordinate (c) 
-- ++(-115:1) coordinate (d) 
-- ++(-89:1) coordinate (e) 
-- ++(-62:1) coordinate (f) 
-- ++(-36:1) coordinate (g) 
-- ++(-10:1) coordinate (h)
-- ++(91:1) coordinate (i)
-- ++(170:1) coordinate (j)
-- ++(118:1) coordinate (k)
-- ++(65:1) coordinate (l)
-- ++(13:1) coordinate (m)
-- ++(39:1) coordinate (n) -- (a);
\draw[ultra thick] (l) -- ++(39:1) coordinate (o) -- (n);
\draw[ultra thick] (o) -- (b);
\draw[ultra thick] (k) -- ++(-89:1) coordinate (p) -- (g);
\draw[ultra thick] (c) -- (l);
\draw[ultra thick] (e) -- (p);
\draw[dashed] (d) -- (k); \draw[dashed] (g) -- (j);
\draw ($.25*(l)+.25*(m)+.25*(n)+.25*(o)$) node {$1$};
\draw ($.25*(a)+.25*(b)+.25*(n)+.25*(o)$) node {$2$};
\draw ($.25*(g)+.25*(h)+.25*(i)+.25*(j)$) node {$3$};
\draw ($.25*(g)+.25*(j)+.25*(k)+.25*(p)$) node {$4$};
\draw ($.25*(b)+.25*(c)+.25*(l)+.25*(o)$) node {$5$};
\draw ($.25*(c)+.25*(d)+.25*(k)+.25*(l)$) node {$6$};
\draw ($.25*(d)+.25*(e)+.25*(k)+.25*(p)$) node {$7$};
\draw ($.25*(e)+.25*(f)+.25*(g)+.25*(p)$) node {$8$};
\end{tikzpicture}$$\end{minipage}
\caption{Demonstration of the injection $R(52143) \hookrightarrow R(6213574)$, as indicated by tile labels. Thick borders in $X(6213574)$ bound the images of the rhombi from the tiling of $X(52143)$, and dashed edges give a (in fact, the \emph{only} for each setting in this example) rhombic tiling of the induced paw.}\label{fig:R mono}
\end{figure}
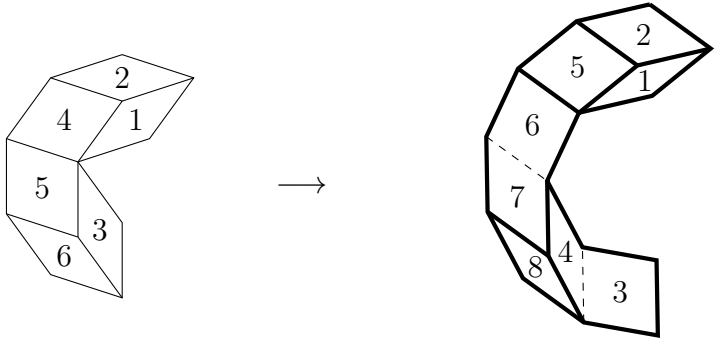

Having established that the numbers of both reduced words and of commutation classes are monotonically increasing with respect to pattern containment, it is natural to wonder when equality is maintained in either case, and if any criteria exist for that to occur. In fact, we can describe such conditions for each statistic, determining precisely which $p \prec w$ preserve the number of reduced words (respectively, commutation classes), and which do not. We prelude those results with a simple proposition.

\begin{prop}\label{prop:1 red wd}
$|R(w)| = 1$ if and only if, for fixed $\varepsilon \in \{\pm 1\}$ and some $m \le m'$,
$$w(i) =
\begin{cases}
i & \text{if } i \not\in [m,m'] \text{, and}\\
i+\varepsilon & \text{otherwise, taking values cyclically in $[m,m']$}.
\end{cases}$$
\end{prop}

\begin{proof}
The only reduced words that support no commutations or braids are words of the form
$${\sf m(m+1)(m+2)\cdots (m'-1)}$$
or
$${\sf (m'-1)(m'-2)(m'-3)\cdots m}$$
for some $m \le m'$.
\end{proof}

Note that if $m = m'$ in Proposition~\ref{prop:1 red wd}, then the reduced word is empty and the permutation $w$ is the identity.

We can now completely characterize the conditions on $p \prec w$ for which $|R(p)|$ is equal to $|R(w)|$. The theorem only addresses the case when these sets have more than one element, since Proposition~\ref{prop:1 red wd} has already described what must occur when they both have size $1$.

\begin{thm}\label{thm:equal R}
Suppose $p \prec w$. Then $|R(p)| = |R(w)| > 1$ if and only if $\ell(p) = \ell(w)$.
\end{thm}

\begin{proof}
Suppose throughout the proof that $p \prec w$, and fix an occurrence $\langle p \rangle$ of $p$ in $w$

Suppose $\ell(p) = \ell(w)$. Then every inversion in $w$ is an inversion in this $\langle p \rangle$, and thus $w$ fixes all letters not in $\langle p \rangle$. Define $x_1 \le y_1 < x_2 \le y_2 < \cdots$ and $P_1, P_2, \ldots$ so that $\langle p \rangle = P_1P_2 \cdots$, where $P_i$ uses the letters $[x_i,y_i]$. Then elements of $R(w)$ are words on
$$\bigcup_i [{\sf x_i},{\sf y_i}),$$
and we have a bijection $R(w) \rightarrow R(p)$ defined letter-wise by
$${\sf j} \mapsto
\begin{cases}
 {\sf j - x_i + 1} & \text{if } x_i \le j < y_i.
\end{cases}$$
Thus $|R(p)| = |R(w)|$.

Now consider $|R(p)| = |R(w)| > 1$. Suppose $\ell(p) < \ell(w)$. Then there is an inversion with values $x>y$ such that, without loss of generality, the value $x$ is not part of $\langle p \rangle$. Choose the positions of this inversion to be as close together as possible, so that everything appearing between $x$ and $y$ in $w$ is both greater than $x$ and part of $\langle p \rangle$. That is,
$$w = A x B y C,$$
where $B \subseteq \langle p \rangle$ and $x < b$ for all $b \in B$.
Because $|R(p)| = |R(w)|$, if we multiply $w$ on the right by adjacent transpositions to remove inversions, then any resulting $w'$ in which $x$ is immediately to the left of $y$ can contain no other descents. In particular, this is true for 
$$w' = AxyBC.$$
Thus $Ax$ is an increasing sequence of values, as is $yBC$. Combining this with the relationship between $x$ and $B$ means that, in fact,
$$AxBC$$
is an increasing sequence of values. Thus $w$ has the form of one of the permutations described in Proposition~\ref{prop:1 red wd}, contradicting the fact that $|R(w)| > 1$. Therefore we must in fact have that $\ell(p) = \ell(w)$.
\end{proof}

Theorem~\ref{thm:equal R} says that any complexity to $w$ besides its containment of $p$ will introduce additional reduced words. To clarify the practicality of this result, we offer the following corollary.

\begin{cor}
Suppose $p \prec w$. Then $|R(p)| = |R(w)| > 1$ if and only if $w$ has a $p$-pattern that can be partitioned as $\langle p \rangle = P_1P_2 \cdots$, where all letters of $P_i$ are less than all letters of $P_{i+1}$ for all $i$, and all letters not in this $\langle p \rangle$ are fixed by $w$.
\end{cor}

Just as reduced word enumeration was shown to be monotonic in Theorem~\ref{thm:R mono} and further refined in Theorem~\ref{thm:equal R}, we can also refine the commutation class monotonicity from Proposition~\ref{prop:C mono}. Whereas Theorem~\ref{thm:R mono} showed that enumeration of $21$-patterns is the crux to equality of $|R(p)|$ and $|R(w)|$, the analogous result for $|C(p)|$ and $|C(w)|$ will rely on enumeration of $321$-patterns.

\begin{thm}\label{thm:equal C}
Suppose $p \prec w$. Then $|C(p)| = |C(w)|$ if and only if the permutations $p$ and $w$ contain the same number of $321$-patterns.
\end{thm}

\begin{proof}
Suppose throughout the proof that $p \prec w$, and fix an occurrence $\langle p \rangle$ of $p$ in $w$.

Suppose that $p$ and $w$ contain the same number of $321$-patterns, so every $321$-pattern in $w$ is part of this $\langle p \rangle$. We want to show that the {\sf MONO} algorithm is both surjective and a function. Consider some $T \in T(p)$. Because $321$-patterns in $w$ occur within $\langle p \rangle$, each $X(q)$-paw that {\sf MONO} produces from a rhombus in $T$ has $q$ avoiding $321$. Thus these paws each have exactly one rhombic tiling. Similarly, in {\sf Step~\ref{monostart}} of {\sf MONO}, the permutation $w_{\ell(p)}$ must also be $321$-avoiding, and thus $|T(w_{\ell(p)})| = 1$. Therefore {\sf MONO} is a function; that is, it maps $T \in T(p)$ to a single, well-defined $T' \in T(w)$. Now consider $U' \in T(w)$. If $p = e$ then $w$ is $321$-avoiding and has only one commutation class, so certainly $|C(p)| = |C(w)|$. Otherwise, there is an $i$ such that $\langle p(i)\rangle > \langle p(i+1) \rangle$. Moreover, because $w$ has no more $321$-patterns than $p$ has, the segment from $\langle p(i)\rangle$ to $\langle p(i+1) \rangle$ in the one-line notation of $w$ is $321$-avoiding. Thus the paw created by this segment has a unique rhombic tiling, which must be what we see in $U'$. Let $w'$ be the permutation obtained by rewriting this segment in increasing order, so $X(w')$ is $X(w)$ without that paw, and let $p' = p\sigma_i$, so $X(p')$ is what remains after positioning a rhombus with edges $\{p(i),p(i+1)\}$ along the rightmost border of $X(p)$. Iterating this procedure with $p'$ and $w'$ yields a unique tiling $U \in T(p)$ such that {\sf MONO} produces $U'$ from $U$. Thus {\sf MONO} is surjective, and so we do indeed have $|C(p)| = |C(w)|$.

Now consider $|C(p)| = |C(w)|$. Suppose that $x > y > z$ is a $321$-pattern in $w$. By Corollary~\ref{cor:isolated paw patterns}, there are elements of $T(w)$ that are identical except for the tiling of a sub-hexagon with edges $\{x,y,z\}$, as depicted in Figure~\ref{fig:sub-hexagons}. 
\begin{figure}[htbp]
\begin{tikzpicture}
\draw (-2.5,1) node {$T' =$};
\draw (4,1) node {$\in T(w)$};
\draw (2,-1.5) to[out=-135,in=0] (1,-2) to[out=180,in=-90] (-2,1) to[out=90,in=180] (1,4) to[out=0,in=135] (2,3.5);
\draw[decorate,decoration={snake}] (2,-1.5) to[out=45,in=-90] (3,1) to[out=90,in=-45] (2,3.5);
\draw (0,0) coordinate (a) -- ++(30:1) coordinate (b) -- ++(90:1) coordinate (c) -- ++(150:1) coordinate (d) -- ++(210:1) coordinate (e) -- ++(270:1) coordinate (f) -- (a);
\draw (d) -- ++(-90:1) coordinate (g) -- (b);
\draw (g) -- (f);
\draw ($(a)!0.5!(b)$) node[below right] {$z$};
\draw ($(c)!0.5!(b)$) node[right] {$y$};
\draw ($(c)!0.5!(d)$) node[above right] {$x$};
\draw ($(e)!0.5!(d)$) node[above left] {$z$};
\draw ($(e)!0.5!(f)$) node[left] {$y$};
\draw ($(a)!0.5!(f)$) node[below left] {$x$};
\end{tikzpicture}
\hspace{.2in}
\begin{tikzpicture}
\draw (-2.5,1) node {$U' =$};
\draw (4,1) node {$\in T(w)$};
\draw (2,-1.5) to[out=-135,in=0] (1,-2) to[out=180,in=-90] (-2,1) to[out=90,in=180] (1,4) to[out=0,in=135] (2,3.5);
\draw[decorate,decoration={snake}] (2,-1.5) to[out=45,in=-90] (3,1) to[out=90,in=-45] (2,3.5);
\draw (0,0) coordinate (a) -- ++(30:1) coordinate (b) -- ++(90:1) coordinate (c) -- ++(150:1) coordinate (d) -- ++(210:1) coordinate (e) -- ++(270:1) coordinate (f) -- (a);
\draw (a) -- ++(90:1) coordinate (g) -- (c);
\draw (g) -- (e);
\draw ($(a)!0.5!(b)$) node[below right] {$z$};
\draw ($(c)!0.5!(b)$) node[right] {$y$};
\draw ($(c)!0.5!(d)$) node[above right] {$x$};
\draw ($(e)!0.5!(d)$) node[above left] {$z$};
\draw ($(e)!0.5!(f)$) node[left] {$y$};
\draw ($(a)!0.5!(f)$) node[below left] {$x$};
\end{tikzpicture}
\caption{Two tilings of the $\{x,y,z\}$ sub-hexagon that must appear in element of $T(w)$, as described in the proof of Theorem~\ref{thm:equal C}.}\label{fig:sub-hexagons}
\end{figure}
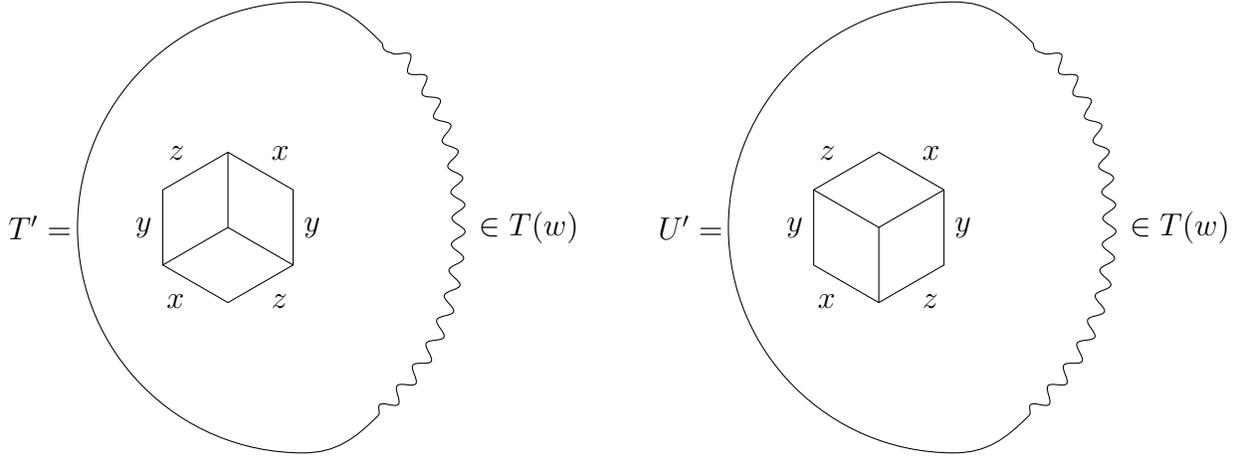
Consider $T' \in T(w)$. For some $T \in T(p)$ to produce this particular $\{x,y,z\}$ sub-hexagon unambiguously via {\sf MONO}, there is an $i$ such that, as defined by the algorithm,
$$w_i = \cdots \ \langle p_i(j)\rangle \ \cdots \ x \ \cdots \ y \ \cdots \ \langle p_i(j+1)\rangle \ z \ \cdots,$$
where $p_i(j) > p_j(j+1)$, and
$$w_{i+1} = \cdots \ y \ x \ z \ \cdots.$$
To obtain this $w_{i+1}$, we must have that $x$ and $y$ are the two largest elements in the segment $\langle p_i(j)\rangle \cdots x \cdots y \cdots \langle p_i(j+1)\rangle$ of $w_i$. For {\sf MONO} to be a function, this segment must avoid $321$. Therefore $y \not> \langle p_i(j+1)\rangle$. Moreover, because $x$ and $y$ are maximal in this segment, we have $y \not< \langle p_i(j+1)\rangle$. Thus $y = \langle p_i(j+1)\rangle$. Because $\langle p_i(j) \rangle > p_j(j+1) = y$, and $x$ is the largest value in the segment, we must in fact have $x = \langle p_i(j) \rangle$. Therefore $x$ and $y$ are both values in $\langle p \rangle$. A similar argument for the hexagon tiling in $U'$ shows that $z$ must be part of $\langle p \rangle$ as well. Therefore every $321$-pattern in $w$ is also a $321$-pattern in $\langle p \rangle$, and so $p$ and $w$ have equally many $321$-patterns.
\end{proof}

Both Theorems~\ref{thm:equal R} and~\ref{thm:equal C} can be stated in terms of patterns in $p$ and $w$: Theorem~\ref{thm:equal R} requires equally many $21$-patterns, and Theorem~\ref{thm:equal C} requires equally many $321$-patterns. This suggests that patterns may play an even deeper, ``meta,'' role in structural aspects of the symmetric group. Additionally, note that this equinumerable requirement is somewhat orthogonal to typical forays into pattern containment and avoidance, where one only cares about \emph{whether} a pattern occurs, and not \emph{how often} it does so.

\subsection{Enumeration of $132$-avoiding permutations by length}\

Permutation patterns appear throughout combinatorics, typically in theorem statements of two varieties: characterization results (such as Corollary~\ref{cor:rdpp main}, because ``vexillary'' is equivalent to $2143$-avoiding) and enumerative results. This latter class has received several decades of consistent attention (see, for example, \cite{atkinson, bona, knuth, simion schmidt, stankova}).

\begin{defn}\label{defn:pattern avoiding set}
Fix a permutation $p$. Let
$$\mathfrak{S}_n(p) = \{w \in \mathfrak{S}_n : w \text{ avoids } p\}.$$
\end{defn}

There is great interest in understanding $|\mathfrak{S}_n(p)|$ for different $p$. For example, it is well-known, and has been proved in many contexts, that
\begin{equation}\label{eqn:s3 avoidance}
p \in \mathfrak{S}_3 \ \Longrightarrow \ |\mathfrak{S}_n(p)| = C_n,
\end{equation}
where
$$C_n = \frac{1}{n+1}\binom{2n}{n}$$
is the $n$th Catalan number. For more information about Catalan numbers, including their enumeration of $\mathfrak{S}_3$ pattern avoidance, the reader is encouraged to read \cite{stanley-catalan} and entry A000108 of \cite{oeis}.

This section will show that partitions and a collection of pattern avoiding permutations are in bijection with each other, and thus are equinumerous. Parts of size $0$ are typically suppressed in partition enumeration -- otherwise there would be infinitely many partitions of a given size. Similarly, it will be useful in the upcoming enumeration of pattern avoidance by length to suppress a certain kind of fixed point in our enumeration of pattern avoidance.

To motivate our convention, consider permutations $v \in \mathfrak{S}_n$ and $v' \in \mathfrak{S}_{n+1}$ defined by
$$v'(x) = \begin{cases}
v(x) & \text{ if } x \le n, \text{ and}\\
n+1 & \text{ if } x = n+1.
\end{cases}$$
If $p \in \mathfrak{S}_k$ with $p(k) \neq k$, then $v \in \mathfrak{S}_n(p)$ if and only if $v' \in \mathfrak{S}_{n+1}(p)$. Moreover, $\ell(v) = \ell(v')$ and $R(v) = R(v')$. In other words, in the context of enumerating $p$-avoidance by length, the permutations $v$ and $v'$ are essentially the same. When $p(1) \neq 1$, an similar statement can be made about $v \in \mathfrak{S}_n$ and $v' \in \mathfrak{S}_{n+1}$ defined by
$$v'(x) = \begin{cases}
1 & \text{ if } x = 1, \text{ and}\\
v(x-1) & \text{ if } x > 1.
\end{cases}$$
Thus, to enumerate pattern avoidance by length, we define a set that sidesteps this overcounting.

\begin{defn}\label{defn:suppression}
For $p \in \mathfrak{S}_k$, set
$$\mathfrak{S}(p) = \bigcup\limits_n\left\{w \in \mathfrak{S}_n : \begin{array}{l}
w \text{ avoids } p,\\
w(n) \neq n \text{ if } p(k) \neq k, \text{ and}\\
w(1) \neq 1 \text{ if } p(1) \neq 1
\end{array}\right\}.$$
\end{defn}

In the applications of this section, we focus on $132$-avoidance. However, other patterns are similarly susceptible to our techniques and results, and we will conclude the paper with a brief indication of how that might proceed.

The pattern $132 \in \mathfrak{S}_3$ appears in both an enumerative result (as described in~\eqref{eqn:s3 avoidance}) and a characterization result (entry P0003 of \cite{dppa}).

\begin{thm}[{\cite{macdonald, manivel}}]
A permutation is $132$-avoiding if and only if it is dominant.
\end{thm}

For information about the significance of dominant permutations, the reader is referred to \cite{macdonald, manivel}. In this paper we will show an enumerative connection between dominant permutations and partitions. In fact this is not the only relationship between these two objects, since a permutation is dominant if and only if its Lehmer code is non-increasing (see, for example, \cite{albert atkinson ruskuc} and \cite{ec1}).

To facilitate our enumeration below, we highlight a crucial aspect of $132$-avoidance that arises from Theorem~\ref{thm:patt redwd}.

\begin{remark}\label{rem:132 as a pattern} 
The permutation $132$ has one reduced word: $R(132) = \{\sf{2}\}$. Moreover, because $132^+ = \{132\}$, Definition~\ref{defn:isolation} and Theorem~\ref{thm:patt redwd} implies that $w$ contains a $132$-pattern if and only if $w$ has a reduced word $\sf{axc}$ for some letter $x$, such that $\sf{a} \in R(u)$ and $\sf{c} \in R(v)$, where $\ell(us_{x-1}) > \ell(u)$ and $\ell(s_{x-1}v) > \ell(v)$.
\end{remark}

We now use the techniques and results of this paper to analyze $132$-avoiding permutations of a given length, and we show -- bijectively -- that these are enumerated by the partition numbers \cite[A000041]{oeis}. The theory of integer partitions is a highly active area of combinatorial research, treated, for example, in the texts \cite{andrews, andrews eriksson} and in a spectacular number of papers.

\begin{defn}
Let $\mathfrak{p}(n) = \big|\{\lambda \vdash n\}\big|$.
\end{defn}

For example, $\mathfrak{p}(4) = 5$ because there are five partitions of $4$. Drawn as Young diagrams anchored at the upper-left corners, these are as follows.

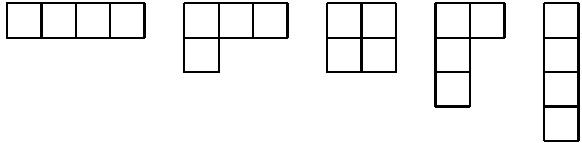
\begin{center}
\hspace{.1in} \tableau{\ &\ &\ &\ } \hspace{.1in} \tableau{\ &\ &\ \\ \ } \hspace{.1in} \tableau{\ &\ \\ \ &\ } \hspace{.1in} \tableau{\ &\ \\ \ \\ \ } \hspace{.1in} \tableau{\ \\ \ \\ \ \\ \ }
\end{center}

In order to study pattern avoidance by length, we consider the following sets, the collection of which partition $\mathfrak{S}(p)$.

\begin{defn}
Fix a permutation $p$ and let
$$\av(p;\ell) = \mathfrak{S}(p) \cap \{w : \ell(w) = \ell\}.$$
\end{defn}

\begin{ex}\
\begin{enumerate}\renewcommand{\labelenumi}{(\alph{enumi})}
\item Some of these sets are finite: $\av(132;4) = \{23451, 51234, 4213, 3241, 3412\}$.
\item Others are infinite:
$$\av(321;3) = \{2341, 4123, 3142, 2413, 23154,21453,31254,21534, \ldots\}.$$
\end{enumerate}
\end{ex}

We conclude this paper with an application of Theorem~\ref{thm:patt redwd}. Theorem~\ref{thm:132 by length} below gives a direct connection between partitions and $132$-avoiding permutations, and recovers a result (proved by other means) of Claesson, Jel\'inek, and Steingr\'imsson \cite[Proposition 11]{claesson jelinek steingrimsson}. Next, we refine Theorem~\ref{thm:132 by length} to give a bijective relationship between restricted partitions and $132$-avoiding permutations with prescribed first letters. Finally, we use this to give a connection to Catalan numbers in Corollary~\ref{cor:refining catalan}, writing those values as sums of quantities of $132$-avoiding permutations.

In light of the multitude and breadth of open problems in the study of pattern avoidance enumeration, it seems highly likely that our techniques can provide new insight into this high-interest field.

\begin{thm}\label{thm:132 by length}
For any integer $\ell \ge 0$,
$$|\av(132;\ell)| = \mathfrak{p}(\ell).$$
\end{thm}

The statement of this result is quite elegant in its simplicity, and its connection between two objects of such significant interest. Moreover, its bijective proof has a variety of implications. Before giving that proof, we discuss a few of those consequences.

We begin by introducing notation used in \cite{ec1}.

\begin{defn}
Let $\mathfrak{p}_k(n)$ be the number of partitions of $n$ into exactly $k$ parts.
\end{defn}

The bijection given below in the proof of Theorem~\ref{thm:132 by length} has the following implication for the number of these restricted partitions.

\begin{cor}\label{cor:restricted partition}
$\mathfrak{p}_k(\ell) = \Big| \av(132;\ell) \cap \{w : w(1) = k+1\}\Big|$.
\end{cor}

Note the connection between these results and \cite[Exercise 1.125]{ec1}. The binary sequences in that setting describe vertical and horizontal steps around the southeast edge of a partition, the inversions in such a sequence are in bijection with the squares of the shape, and the number of $1$s refers to the height (or width, depending on convention) of the shape. Thus, in light of our results here, these strings can also describe classes of $132$-avoiding permutations.

Another way to refine the result of Theorem~\ref{thm:132 by length} involves the following sets.

\begin{defn}
Partition $\av(p;\ell)$ by defining
$$\av(p;\ell, d) = \av(p;\ell) \cap \{w : \text{elements of } R(w) \text{ have $d$ distinct letters}\}.$$
\end{defn}

\begin{ex}\
\begin{enumerate}\renewcommand{\labelenumi}{(\alph{enumi})}
\item $\av(132;4,3) = \{4213, 3241, 3412\}$, whereas $\av(132;4,4) = \{23451, 51234\}$.
\item $\av(321;3,2) = \emptyset$ and $\av(321;3,3) = \av(321;3)$.
\end{enumerate}
\end{ex}

The bijection in the proof of Theorem~\ref{thm:132 by length} shows that
$$\big|\av(132;\ell,d)\big| = \big|\{\lambda \vdash \ell : \lambda \text{ fits inside the staircase shape } \delta_{d+1} \text{ but not inside } \delta_d\}\big|.$$
These values for small $\ell$ and $d$ are given in Table~\ref{table:fixed length and support}, and were obtained using Sage.

\begin{table}[htbp]
\begin{tabular}{c|cccccccccccc}
& $\ell = 0$ & 1 & 2 & 3 & 4 & 5 & 6 & 7 & 8 & 9 & 10 & 11\\
\hline
$d=0$ & 1 &0&0&0&0&0&0&0&0&0&0&0\\
1 &0 & 1 &0&0&0&0&0&0&0&0&0&0\\
2 &0&0&2 & 1&0&0&0&0&0&0&0&0\\
3 &0&0&0&2&3&3&1&0&0&0&0&0\\
4 &0&0&0&0&2&2&6&7&6&4&1&0\\
5 &0&0&0&0&0&2&2&4&8&12&15&17\\
6 &0&0&0&0&0&0&2&2&4&6&12&15\\
7 &0&0&0&0&0&0&0&2&2&4&6&10\\
8 &0&0&0&0&0&0&0&0&2&2&4&6\\
9 &0&0&0&0&0&0&0&0&0&2&2&4\\
10 &0&0&0&0&0&0&0&0&0&0&2&2\\
11 &0&0&0&0&0&0&0&0&0&0&0&2
\end{tabular}
\vspace{.1in}
\caption{The values $\big|\av(132;\ell,d)\big|$ for $0 \le \ell, d \le 11$.}\label{table:fixed length and support}
\end{table}

Because the sets $\av(132;\ell,d)$ partition $\av(132;\ell)$, we have
$$\big|\av(132;\ell)\big| = \sum_{d=0}^{\ell} \big| \av(132;\ell,d)\big|.$$
Moreover, we know from \eqref{eqn:s3 avoidance} that the Catalan numbers can be refined by this statistic $\big|\av(132;\ell,d)\big|$.

\begin{cor}\label{cor:refining catalan}
$$C_n = \sum_{0 \le d < n \atop 0 \le \ell \le \binom{n}{2}} \big|\av(132;\ell,d)\big|.$$
\end{cor}

In other words, $C_n$ is the sum of all values weakly northwest of the entry $\ell = \binom{n}{2}$ and $d = n-1$ in the extension of Table~\ref{table:fixed length and support}.

\begin{ex}
\begin{eqnarray*}
\sum_{0 \le d < 4 \atop 0 \le \ell \le \binom{4}{2}} \big|\av(132;\ell,d)\big| &=& \sum_{0 \le d < 4} \ \sum_{0 \le \ell \le 6} \big|\av(132;\ell,d)\big|\\
&=& (1) + (1) + (2 + 1) + (2 + 3 + 3 + 1)\\
&=& 14\\
&=& C_4.
\end{eqnarray*}
\end{ex}

We conclude this section with the proof of Theorem~\ref{thm:132 by length}. We divide the argument into a sequence of lemmas in order to highlight the features of the bijection involved in the proof of the theorem.

Throughout our arguments, we consider partitions to be tableaux with a particular filling.

\begin{defn}
The \emph{antidiagonal filling} of a partition $\lambda$ is defined so that all squares along the highest antidiagonal (that is, southwest-to-northeast) are labeled $1$, all squares along the next highest antidiagonal are labeled $2$, and so on. Equivalently, if the rows are indexed from top to bottom and the columns are indexed from left to right, then the square in row $r$ and column $c$ is labeled $r+c-1$.
\end{defn}

An example of this filling appears in Figure~\ref{fig:antidiagonal}. 

\begin{figure}[htbp]
$\tableau{1&2&3&4&5&6&7\\2&3&4&5\\3&4&5&6\\4&5\\5}$
\caption{Antidiagonal filling of the partition $(7,4,4,2,1)\vdash 18$.}\label{fig:antidiagonal}
\end{figure}

\begin{defn}
For any partition $\lambda$ with the antidiagonal filling, its \emph{reading word}, denoted $\read(\lambda)$, is obtained by concatenating the entries of the rows of $\lambda$, from bottom to top, and from left to right within each row.
\end{defn}

Continuing the example depicted in Figure~\ref{fig:antidiagonal}, $\read(7,4,4,2,1) = \sf{545345623451234567}$.

\begin{lem}\label{lem:reading word is reduced}
For any partition $\lambda$ with the antidiagonal filling, the word $\read(\lambda)$ is reduced.
\end{lem}

\begin{proof}
To be unreduced, we must be able to apply commutation and braid moves to this word to obtain a new word with either $xx$ or, without loss of generality, $x(x-1)x(x-1)$ as factors. Antidiagonal fillings impose strict rules on $\read(\lambda)$. In particular, any two appearances of $x$ in $\read(\lambda)$ are separated by an appearance of $x-1$. Neither $x$ can pass over this $x-1$ to reach the other, so we cannot encounter an $xx$ factor. Similarly, we cannot see $\cdots x \cdots x-1 \cdots x \cdots x-1 \cdots$ without $x-2$ appearing between the two copies of $x-1$, and thus the factor $x(x-1)x(x-1)$ does not arise either.
\end{proof}

Lemma~\ref{lem:reading word is reduced} means that $\read(\lambda)$ defines a permutation. Moreover, the length of this permutation is equal to the size of $\lambda$, because it is equal to the number of letters in $\read(\lambda)$.

\begin{defn}
Write $\pi(\lambda)$ for the permutation for which $\read(\lambda)$ is a reduced word.
\end{defn}

Continuing our example, $\read(\lambda) \in R(65472381)$. Thus $\pi(7,4,4,2,1) = 65472381 \in \mathfrak{S}_8$. Note that this permutation is $132$-avoiding.

\begin{lem}\label{lem:reading word is lex greatest}
The map $\lambda \mapsto \pi(\lambda)$ is injective.
\end{lem}

\begin{proof}
Let $\lambda$ be a partition having $k$ parts, and $k$th part of size $i$. Suppose that $\pi(\lambda) = \pi(\lambda')$. By construction, the permutation $\pi(\lambda)$ sends $1$ to $k+1$. Thus, $\lambda$ and $\lambda'$ must have the same number of rows. If $\lambda \neq \lambda'$, then it suffices to assume that there bottom rows have different lengths, say $i$ and $i'$, respectively, where $i<i'$. By definition of the filling and construction of the permutation, the leftmost values in the one-line notation of $\pi(\lambda)$ are $(k+1)(k+2)\cdots(k+i)$. Moreover, the next value, the image of $i+1$, must be less than $k+1$. The leftmost values in the one-line notation of $\pi(\lambda)$ are $(k+1)(k+2)\cdots(k+i')$. Since $i' > i$, this is a contradiction. Therefore $\lambda = \lambda'$.
\end{proof}

We now show the connection between $132$-patterns and the permutations $\{\pi(\lambda)\}$.

\begin{lem}\label{lem:132 bijection}
$\mathfrak{S}(132) = \{\pi(\lambda)\}$.
\end{lem}

\begin{proof}
Fix $w \in \mathfrak{S}(132)$ and let $\sf{s} \in R(w)$ be its lexicographically greatest reduced word. View $\sf{s}$ as the concatenation of maximally long factors of consecutively increasing values:
$$\sf{\big(x_1 (x_1+1) (x_1+2) \cdots (x_1 + i_1)\big)\big(x_2 (x_2+1) (x_2+2) \cdots (x_2 + i_2)\big) \cdots}.$$
Suppose that $x_{j'+1} = x_{j'} - 1$ for all $j' < j$, and consider $x_{j+1}$. The word is reduced, so $x_{j+1} \neq x_j+i_j$. Also, $i_j$ is maximal, so $x_{j+1} \neq x_j+i_j+1$. If $x_{j+1} > x_j+s_j+1$, then $\sf{s}$ would not be maximal because $x_{j+1}$ and $x_j+i_j$ could commute. Similarly, if $x_{j+1} \in [x_j,x_j+i_j-1]$, then commutation moves to slide $x_{j+1}$ leftward and a braid move to change $x_{j+1}(x_{j+1}+1)x_{j+1}$ into $(x_{j+1}+1)x_j(x_{j+1}+1)$ would again contradict the maximality of $\sf{s}$. Thus $x_{j+1} < x_j$. Define $u$ and $v$ as in Definition~\ref{defn:isolation}, so that $u\sigma_{x_j}v = w$. The choice of $j$ means that $\ell(u\sigma_{x_j-1}) > \ell(u)$. To be $132$-avoiding, the letter $x_j$ cannot be isolated on $[x_j-1,x_j]$ (see Remark~\ref{rem:132 as a pattern}). Thus, $\ell(\sigma_{x_j-1}v) < \ell(v)$. Because $\sf{s}$ is lexicographically maximal, $x_{j+1} = x_j-1$ for all $j$.

Suppose that $i_j < i_{j-1}$. Obtain $\sf{s'}$ from $\sf{s}$ by sliding the factor $(x_{j-1}+i_j+1)\cdots (x_{j-1}+i_{j-1})$ to the right of $x_j+i_j = x_{j-1}-1+i_j$ via a sequence of commutations. In $\sf{s'}$, the letter $x_{j-1}+i_j+1$ is isolated on $[x_{j-1}+i_j,x_{j-1}+i_j+1]$, meaning that $w \not\in \mathfrak{S}(132)$ -- a contradiction. Finally, if a nonempty $\sf{s}$ has minimal letter $m > 1$, then any appearance of $m$ would be isolated on $[m-1,m]$. Thus, by the previous discussion, $x_J = 1$ for the maximal value of $J$. The word $\sf{s}$ as described is exactly $\read(i_J+1, i_{J-1}+1, \ldots, i_1+1)$. Thus $\mathfrak{S}(132) \subseteq \{\pi(\lambda)\}$.

For the reverse inclusion, first suppose that $\lambda = (\ell)$ has only one part. Then $\read(\lambda) = \sf{12\cdots \ell}$ and $\pi(\lambda) = 234\cdots \ell(\ell+1)1$, which is certainly $132$-avoiding. Suppose, inductively, that for any partition $\mu$ having fewer than $k$ parts, the permutation $\pi(\mu)$ is $132$-avoiding. Let $\lambda$ be a partition having $k$ parts and $k$th part of size $i$. Let $\mu$ be the partition obtained from $\lambda$ by removing the last part. By construction, $\pi(\lambda) = u\cdot\pi(\mu)$, where $u = \sigma_{k} \sigma_{k+1}\cdots \sigma_{k+i-1}$. The leftmost values in the one-line notation of $\pi(\lambda)$ are $(k+1)(k+2)\cdots (k+i)$. The only inversions in $\pi(\lambda)$ that do not appear in $\pi(\mu)$ arise because $k$ now appears to the right of the values $\{k+1,k+2,\ldots, k+i\}$. For $\pi(\lambda)$ to introduce a $132$-pattern, then, we would need $\langle 32\rangle = (k+j)k$ for some $j \in [1,i]$. However, no value less than $k$ appears to the left of this $k+j$, so there can be no such $132$-pattern. Hence $\pi(\lambda) \in \mathfrak{S}(132)$ and so $\{\pi(\lambda)\} \subseteq \mathfrak{S}(132)$.
\end{proof}

We are now ready to prove the main result of this section.

\begin{proof}[Proof of Theorem~\ref{thm:132 by length}]
Lemmas~\ref{lem:reading word is reduced} and~\ref{lem:reading word is lex greatest} give a bijective correspondence between partitions and a particular class $\{\pi(\lambda)\}$ of permutations. Lemma~\ref{lem:132 bijection} shows that this class is exactly $\mathfrak{S}(132)$. Therefore we have a bijection between partitions and elements of $\mathfrak{S}(132)$, and, more precisely, between partitions of $\ell$ and elements of $\mathfrak{S}(132;\ell)$.
\end{proof}

By the symmetry of their reduced words, this section's manipulations and enumerative results about $132$-avoidance can be recast in terms of $213$-avoidance. However, despite expression \eqref{eqn:s3 avoidance}, we cannot replace $132$ there by any other element of $\mathfrak{S}_3$ in those statements. Our techniques would still apply for those other patterns, but the details -- and the enumerations -- would differ. In the final section of this paper, we give a glimpse of how this would work, and state enumerative results about $231$-avoidance.

\section{Final remarks}

It is clear from this work that reduced word manipulation is ripe for application in many settings. Here we have focused only on the symmetric group, but the technique could certainly be employed in other Coxeter groups and to other structural questions. The breadth of applications we have encountered so far is highly promising, and we are optimistic that these methods will continue to yield fruitful results.

For example, if one were to enumerate $231$-avoiding permutations (that is, stack-sortable permutations, \cite{knuth}) by length, one would be looking for factors $x(x+1)$ to be unisolated on $[x,x+1]$. This means, among other things, that if the letters $x$ and $x+1$ each appear only once in a reduced word, then their relative order is forced: they would have to appear as $\sf{(x+1)}\cdots \sf{x}$. On the other hand, if $x$ and $y$ each appear only once in a reduced word and $|x-y| > 1$, then their relative order does not matter. This produces the following enumerations of $231$-avoiding permutations in $\mathfrak{S}_n$ by length.
\begin{eqnarray*}
\big|\{w \in \mathfrak{S}_n(231) : \ell(w) = 1\}\big| &=& n-1\\
\big|\{w \in \mathfrak{S}_n(231) : \ell(w) = 2\}\big| &=&\binom{n-1}{2}\\
\big|\{w \in \mathfrak{S}_n(231) : \ell(w) = 3\}\big| &=&\binom{n-1}{3} + n-2\\
\big|\{w \in \mathfrak{S}_n(231) : \ell(w) = 4\}\big| &=&\binom{n-1}{4} + (n-2)(n-3)\\
\big|\{w \in \mathfrak{S}_n(231) : \ell(w) = 5\}\big| &=&\binom{n-1}{5} + (n-2)\binom{n-3}{2} + n-3
\end{eqnarray*}

\section*{Acknowledgements}

I am grateful for the thoughtful comments of the anonymous referees.

\end{document}